\documentclass[a4paper,reqno]{amsart}

\usepackage{amsmath, amssymb, amsthm}
\usepackage{enumitem}
\usepackage{hyperref}
\hypersetup{
  colorlinks=false,
  pdfborder={0 0 0}
}

\newtheorem{theorem}{Theorem}[section]
\newtheorem{lemma}[theorem]{Lemma}

\newtheorem{proposition}[theorem]{Proposition}

\theoremstyle{definition}
\newtheorem{definition}[theorem]{Definition}

\theoremstyle{definition}
\newtheorem{remark}[theorem]{Remark}

\theoremstyle{definition}
\newtheorem{example}[theorem]{Example}

\theoremstyle{definition}
\newtheorem{assumption}[theorem]{Assumption}

\theoremstyle{definition}

\theoremstyle{definition}
\newtheorem{notation}[theorem]{Notation}

\theoremstyle{plain}
\newtheorem{mainresult}{Main Result}

\makeatletter

\@addtoreset{equation}{subsection}
\makeatother

\makeatletter
\renewenvironment{proof}[1][\proofname]{%
  \par\pushQED{\qed}%
  \normalfont%
  \topsep6\p@\@plus6\p@\relax
  \trivlist
  \item[\hskip\labelsep\textbf{#1\@addpunct{.}}]%
}{%
  \popQED\endtrivlist\@endpefalse
}
\makeatother

\newcommand{\supp}{\operatorname{supp}}

\newcommand{\U}{\mathbf{U}}
\makeatletter
\let\oldtocsection=\tocsection
\renewcommand{\tocsection}[3]{%
  \oldtocsection{#1}{#2}{\textbf{#3}}%
}
\let\oldtocsubsection=\tocsubsection
\renewcommand{\tocsubsection}[3]{%
  \oldtocsubsection{#1}{#2}{\hspace*{0.25em}#3}%
}
\makeatother
\author{Yoshiyuki Endo}
\address{Department of Mathematics, Nagoya University}
\email{endo.yoshiyuki.e2@s.mail.nagoya-u.ac.jp}
\title[Julia--Fatou Theory via RSCCs]
{A Julia--Fatou Theory via Random Systems with Complete Connections}

\subjclass[2020]{37F10, 37A50, 60J05}
\keywords{random dynamical systems,
random systems with complete connections (RSCC),
Julia sets and Fatou sets,
kernel Julia sets,
Cooperation Principle}

\begin{document}

\begin{abstract}
We develop a Julia--Fatou theory for random dynamical systems of continuous self-maps on compact metric spaces driven by random systems with complete connections (RSCCs).

To each index, we assign either a family of continuous self-maps or a probability measure supported on such a family. Together with these assigned objects, a random index sequence gives rise to a random dynamical system on the compact metric space. Although the index process need not be a finite-order Markov chain, the recursively updated state process is a time-homogeneous Markov chain. For each state, we decompose the compact metric space into the Julia and Fatou sets. We also define the corresponding kernel Julia set and Julia sets along admissible paths.

We further study the averaged dynamics on the product of the compact metric space and the state space and prove Cooperation Principle~I. Under suitable assumptions, emptiness of the kernel Julia set at every state implies equicontinuity of the iterates of the adjoint transition operator on the whole space of probability measures. Moreover, for every initial state and finite Borel measure, the Julia set has measure zero along almost every admissible path. This extends previous results for graph-directed Markov systems with finitely many vertices and edges.

We also identify a phenomenon specific to RSCCs, namely emptiness jumps of kernel Julia sets along admissible state trajectories. Such a jump may cause the kernel-emptiness hypothesis, which is used as a sufficient condition in Cooperation Principle~I, to fail. We provide several conditions excluding these jumps. Examples motivated by mechanisms such as reinforcement illustrate both the jump phenomenon and the applicability of Cooperation Principle~I to state-dependent selection rules.
\end{abstract}

\maketitle
\tableofcontents

\section{Introduction}

\subsection{Background and Motivation}
\leavevmode
The study of complex dynamics has its origins in the works of Julia and Fatou, where one studies the
iteration of a single holomorphic map and the dichotomy between stable and unstable
behavior, encoded by the Fatou and Julia sets
\cite{JMPA_1918_8_1__47_0,BSMF_1919__47__161_0}; see, for instance,
\cite{MR1128089,MR1230383,MR2193309} for standard references.
A natural extension is to consider compositions of multiple maps.
This viewpoint was developed in the theory of rational semigroups initiated by
Hinkkanen and Martin \cite{MR1397693,MR1429333}.
Randomization adds another layer. At each step, the map is chosen according to a probability law.
Random complex dynamics was initiated by Forn{\ae}ss and Sibony \cite{MR1145616} and further
developed by Sumi and others; see, for instance,
\cite{MR2323605,MR2747724,MR3084426,MR3333714,MR3382590,MR3649239,MR4002398,MR4268827,MR4407234}.
Sumi introduced a family of cooperation principles in random complex dynamics.  In particular,
Cooperation Principle~I, established in \cite{MR2747724} for i.i.d.\ random dynamical systems,
asserts that if the kernel Julia set is empty, then the ``chaos'' of the averaged dynamics disappears.
It is worth noting that this is a randomness-induced phenomenon, which does not arise in deterministic dynamics.

\smallskip
A further step beyond the i.i.d.\ setting is to allow dependent choices of maps.  In \cite{MR4002398},
Sumi and Watanabe extended Cooperation Principle~I to non-i.i.d.\ holomorphic random dynamics
driven by a Markov chain.
Their framework is formulated via graph-directed Markov systems (GDMSs), introduced by Mauldin--Urba\'nski, although the maps in the systems considered in \cite{MR4002398}
are not assumed to be contractive, and it contains i.i.d.\ random dynamical systems as a special case.
The present paper extends this framework by allowing the selection rule
to be governed by a recursively updated state in a general state space,
rather than a finite vertex set. The resulting index process need not
be a finite-order Markov chain, while the state process remains
Markovian.

\smallskip
We work in the probabilistic framework of random systems with complete connections
(RSCCs). This framework is rooted
in the notion of dependence with complete connections introduced by
Onicescu and Mihoc in 1935 \cite{OnicescuMihoc1935}, 
further developed by Doeblin and Fortet \cite{MR1505076}, 
and systematically studied by Iosifescu and Grigorescu \cite{MR1070097}.

\smallskip
An RSCC specifies the conditional law of the next index through a state that is updated from the preceding state and the selected index.
This state represents the dependence of future choices on preceding indices. It need not be determined by any fixed finite number of the most recent indices.
Owing to this generality,
RSCCs appear in a variety of mathematical contexts.
In fractal geometry and ergodic theory,
the work \cite{MR4396787} develops a unified framework encompassing
countable iterated function systems with overlaps,
Smale endomorphisms, and
RSCCs, establishing geometric
and dimensional properties of stationary measures.
Connections with statistical mechanics have also been explored:
the works \cite{MR2041831,MR2123648} recast RSCCs in a framework
parallel to Gibbs theory, thereby clarifying
the thermodynamic features of these systems.
In number theory,
RSCCs have
been applied to continued fraction expansions
\cite{MR3035955,MR3523394,MR4075402}.
The digit sequence of a continued fraction is not a finite-order
Markov chain but can naturally be interpreted as a chain with complete
connections, fitting into the RSCC framework.
Furthermore, in time series analysis,
the work \cite{MR4140544} provides general conditions ensuring stationarity,
ergodicity, and mixing properties for infinite-memory processes.
The framework covers observation-driven models with exogenous covariates
arising in finance, economics, and climate data,
thus supplying theoretical foundations for models used in practice.
However, to the best of the author's knowledge,
RSCCs have not been systematically employed as a foundation for
a Julia--Fatou theory in complex dynamics
or, more generally, for dynamical systems of continuous maps
on compact metric spaces.

\smallskip
This paper develops a Julia--Fatou theory for random dynamical systems of continuous self-maps on compact metric spaces, driven by RSCCs.
We extend the finite-state GDMS framework by defining admissible compositions and Julia, Fatou, and kernel Julia sets at each state of
an RSCC. For compact metric state spaces, accumulation of states allows emptiness jumps of kernel Julia sets.
To establish the cooperation principle, we impose equicontinuity
of the iterated update maps and total-variation continuity of the
finite-word distributions, uniformly in the word length.
Our strategy is to extend the GDMS viewpoint of \cite{MR4002398} while preserving the
equicontinuity characterization that underlies the classical Julia--Fatou dichotomy.

\subsection{Setting and Main Results}
We briefly describe the setting and main results. Precise definitions and standing
assumptions are given in the subsequent sections.

Throughout, let $Y$ be a nonempty compact metric space. Let $\mathrm{CM}(Y)$ denote
the space of continuous self-maps on $Y$ endowed with the compact--open topology.
Equivalently, this is the topology of uniform convergence on $Y$.
Our primary motivation comes from complex dynamics on the Riemann sphere
$\widehat{\mathbb C}$.
In that classical case one takes $Y=\widehat{\mathbb C}$.
Here we work on a general compact metric space $Y$, keeping only the
topological structure needed for a Julia--Fatou theory.

We work with
an RSCC consisting of
a state space $(W,\mathcal W)$, an index space $(X,\mathcal X)$,
an update map $u:W\times X\to W$, and a transition probability function
$P:W\times\mathcal X\to[0,1]$; see Definition~\ref{def:RSCC}.
By Theorem~\ref{thm:RSCC-existence}, each initial state $w\in W$
determines a unique probability measure $\mathbf P_w$ on
$(X^{\mathbb N},\mathcal X^{\mathbb N})$.
Under this measure, let $(\xi_n)_{n\in\mathbb N}$ be the associated
index process, and set
\[
\zeta_0=w,\qquad
\zeta_n=u(\zeta_{n-1},\xi_n)\quad(n\ge1).
\]
Then
\[
\mathbf P_w(\xi_{n+1}\in A\mid\xi_1,\ldots,\xi_n)
=P(\zeta_n,A),
\qquad A\in\mathcal X,\quad n\ge1,
\]
$\mathbf P_w$-almost surely.
Thus the conditional law of the next index
can depend on all preceding indices through $\zeta_n$.
The index process $(\xi_n)_{n\in\mathbb N}$ need not be a finite-order Markov chain;
see Example~\ref{ex:linear-reinforcement-rscc}.
In contrast, $(\zeta_n)_{n\in\mathbb N_{0}}$ is a time-homogeneous Markov chain with
transition kernel
\[
Q(v,B)=\int_X\mathbf1_B(u(v,x))\,P(v,dx),
\qquad v\in W,\quad B\in\mathcal W.
\]
Equivalently, the process
$Z_n=(\zeta_{n-1},\xi_n)$, $n\ge1$, is a time-homogeneous
Markov chain on $W\times X$, and $\xi_n$ is its
second-coordinate observation.
Thus the index process admits a hidden-Markov-type representation,
although it need not itself be Markov.

To each index $x\in X$ we associate a Borel probability measure
$\tau_x$ on $\mathrm{CM}(Y)$ and write $\Gamma_x:=\supp(\tau_x)$.
Given $y\in Y$, at step $n$ we choose the index $\xi_n$ according to the RSCC rule.
We then choose a map $\gamma_n\in\Gamma_{\xi_n}$ according to $\tau_{\xi_n}$ and
apply it to the current point in $Y$.
This procedure defines the random dynamical system $S_\tau$ on $Y$;
see Definition~\ref{def:RSCC-tau}.
Writing $y_0=y$ and $y_n=\gamma_n\circ\cdots\circ\gamma_1(y)$,
the joint process $(y_n,\zeta_n)_{n\in\mathbb N_{0}}$ is time-homogeneous Markov on $Y\times W$.
Its transition operator is $M_\tau$, introduced in
Section~\ref{sec:generalization-CP}.
The cooperation principle concerns this Markov evolution on the product
space and the Julia sets along the resulting admissible paths.
When the discussion depends only on the supports $\Gamma_x$ and not on the measures
$\tau_x$ themselves, we write $S$ in place of $S_\tau$; see
Definition~\ref{def:basic-setup}.
We refer to both $S_\tau$ and $S$ as \emph{RSCCs on $Y$}.

For each state $w\in W$, we consider the family $H_w(S)$ of all admissible finite
compositions starting from $w$; see
Definition~\ref{def:admissible-composition-family}.
We define the Fatou and Julia sets $F_w(S)$ and $J_w(S)$ by equicontinuity of
$H_w(S)$; see Definition~\ref{def:Julia-RSCC}.
To capture persistent instability, we also introduce the kernel Julia set
$J_{\ker,w}(S)$, which consists of points that are mapped into Julia sets under
every admissible future evolution; see Definition~\ref{def:kernel-julia-set}.
These notions extend the ones studied in \cite{MR2747724,MR4002398}.
In particular,
finite graph-directed Markov systems form
a special case; see Example~\ref{ex:rscc-gdms-mrds}.

In Section~\ref{sec:generalization-CP} we work under
Assumption~\ref{ass:standing-cooperation}.
In particular, the state space $W$ is a compact metric space and the
index space $X$ is finite. We further assume equicontinuity of the
iterated update maps and total-variation continuity of the finite-word
distributions with respect to the initial state, uniformly in the word
length. The latter two conditions imply continuity of the one-step
update maps and transition probabilities, respectively.
We introduce the product space $\mathbb{Y}:=Y\times W$ and study Markov operators
on $C(\mathbb{Y})$, the Banach space of complex-valued continuous functions on
$\mathbb{Y}$, together with their adjoint action on $\mathfrak{M}_1(\mathbb{Y})$.
Here $\mathfrak{M}_1(\mathbb{Y})$ denotes the space of Borel probability measures
on $\mathbb{Y}$ endowed with the weak-$*$ topology, which is compact and metrizable.
We then define the transition operator $M_\tau$ associated with $S_\tau$
as in Definition~\ref{def:transition-operator} and derive basic properties.

The Cooperation Principle~I is formulated in terms of three objects.
First, the adjoint operator $M_\tau^{*}$ acts on $\mathfrak{M}_1(\mathbb{Y})$, and
$F_{\mathrm{meas}}(M_\tau^{*})$ denotes the set of measures at which the iterates
$\{(M_\tau^{*})^{n}\}_{n\in \mathbb{N}}$ are equicontinuous on some neighborhood; see
Definition~\ref{def:Fatou-meas}.
Second, for each initial state $w\in W$, the RSCC together with
$\tau=\{\tau_x\}_{x\in X}$ induces a natural probability measure $\tilde{\tau}_w$
on the path space $\Xi_w(S_\tau)$; see Definition~\ref{def:Pwtilde}.
Third, for $\xi\in\Xi_w(S_\tau)$ we write $J_\xi$ for the Julia set along $\xi$;
see Definition~\ref{def:Julia-path}.
With this notation, we can state the Cooperation Principle~I in the RSCC setting.

\begin{mainresult}[Theorem~\ref{thm:CP-Fatou},
Cooperation Principle~I for RSCCs]
\label{mr:cooperation}
Let $S_\tau$ satisfy Assumption~\ref{ass:standing-cooperation}.
If $J_{\ker,w}(S_\tau)=\emptyset$ for all $w\in W$, then
\(
F_{\mathrm{meas}}(M_\tau^{*})=\mathfrak{M}_1(\mathbb{Y}).
\)
Moreover, for every finite Borel measure $\lambda$ on $Y$
and every $w\in W$, we have $\lambda(J_\xi)=0$
for $\tilde{\tau}_w$-almost every $\xi\in\Xi_w(S_\tau)$.
\end{mainresult}

This extends Cooperation Principle~I to RSCCs whose index processes
need not be finite-order Markov chains.
For each such RSCC, the joint process on $Y\times W$ remains Markovian,
and its transition operator $M_\tau$ governs the averaged dynamics
considered here.

In our approach, we use the assumption $J_{\ker,w}(S_\tau)=\emptyset$ for all $w\in W$
as a sufficient condition for Cooperation Principle~I.
When the state space has an accumulation point, kernel-emptiness need not
be stable under limits of admissible state trajectories.
The next result shows that, on every metric state space having an
accumulation point, the remaining RSCC data can be chosen so that this
failure of stability occurs.

\begin{mainresult}[Emptiness jumps of kernel Julia sets]
\label{mr:jump}
Let $W$ be a metric space having an accumulation point, and set
$\mathcal W:=\mathcal B(W)$.
Then there exist a finite set $X$ with
$\mathcal X:=\mathcal P(X)$, a measurable update map
$u:W\times X\to W$, a transition probability function
$P:W\times\mathcal X\to[0,1]$, and a family
$\{\Gamma_x\}_{x\in X}$ of nonempty subsets of
$\mathrm{CM}(\widehat{\mathbb C})$ such that the resulting system,
denoted by $S$, is an RSCC on $\widehat{\mathbb C}$ and has the
following property.

There exist states $w_0,w_\infty\in W$ and an index sequence
$(x_n)_{n\in\mathbb N}\in\operatorname{supp}(\mathbf P_{w_0})$
such that the state sequence $(w_n)_{n\in\mathbb N_0}$ defined by
$w_{n+1}:=u(w_n,x_{n+1})$ for $n\in\mathbb N_0$ satisfies
\[
w_n\longrightarrow w_\infty,
\qquad
J_{\ker,w_n}(S)=\emptyset\quad(n\in\mathbb N),
\qquad
J_{\ker,w_\infty}(S)\neq\emptyset.
\]
\end{mainresult}

We call the loss of kernel-emptiness displayed above an
\emph{emptiness jump of kernel Julia sets}.
This notion is formalized in Definition~\ref{def:kernel-julia-jump}.
Example~\ref{ex:kernel-julia-jump} gives an explicit construction, while
Remark~\ref{rem:jump-arbitrary-state-space} states its extension to every
metric state space having an accumulation point.
A reinforcement-type example is given in
Example~\ref{ex:linear-reinforcement-rscc}.

\begin{mainresult}[Proposition~\ref{prop:no-jump-discrete-W}, Theorem~\ref{thm:no-jump-phi-irreducible}, Criteria excluding emptiness jumps]
\label{mr:nojump}
Let $S$ be an RSCC on $Y$. 
The jump phenomenon in Definition~\ref{def:kernel-julia-jump} does not occur if one of the
following conditions holds:
\begin{enumerate}[label=\textup{(\roman*)}]
\item
The state space $W$ is discrete and $\mathcal W=\mathcal B(W)$.

\item
There exists a $\sigma$-finite measure $\varphi$ on $(W,\mathcal W)$ such that $S$ is
$\varphi$-irreducible in the sense of Definition~\ref{def:phi-irreducible-RSCC}, and
\[
\varphi\bigl(\{w\in W:\ J_{\ker,w}(S)=\emptyset\}\bigr)>0.
\]
\end{enumerate}
\end{mainresult}

The jump phenomenon is a distinctive feature of the RSCC framework.
It cannot occur in the classical GDMS setting, since a GDMS with
a discrete vertex set
corresponds to an RSCC with a discrete state space, and hence condition
\textup{(i)} applies.

Condition \textup{(ii)} provides another mechanism excluding jumps.
If $S$ is $\varphi$-irreducible and kernel-emptiness holds on a set of positive
$\varphi$-measure, then kernel-emptiness holds at every state.
In particular, no jump can occur.

\subsection{Organization of the Paper}
\leavevmode
Section~\ref{sec:formulation} develops the RSCC-based Julia--Fatou
framework under the assumption that the index space is countable.
We first introduce admissible index words, reachable states, and
admissible composition families. We then define statewise Julia, Fatou,
and kernel Julia sets and establish their basic properties.
We also study the propagation of kernel-emptiness and the emptiness
jump phenomenon, and finally define Julia sets along admissible paths
and construct the associated path-space probability measures.
Section~\ref{sec:generalization-CP} is developed under
Assumption~\ref{ass:standing-cooperation}, which requires a compact
state space and a finite index space. It further requires
equicontinuity of the iterated update maps and total-variation
continuity of the finite-word distributions with respect to the
initial state, uniformly in the word length. The latter two
conditions imply continuity of the one-step update maps and
transition probabilities, respectively.
Under these assumptions, we study Markov operators on the product space
$Y\times W$, introduce the transition operator describing the averaged
dynamics and its iterates, and prove Cooperation Principle~I in the
RSCC setting.
Section~\ref{sec:examples} presents a reinforcement model exhibiting
an emptiness jump, a modified model satisfying the hypotheses
of Cooperation Principle~I, and a general feedback construction.

\medskip
To avoid confusion, we fix some basic notation used throughout the paper.
We use the standard notation
$\mathbb{N}=\{1,2,3,\dots\}$ and $\mathbb{N}_0=\{0,1,2,3,\dots\}$, as well as
the standard symbols $\mathbb Z$, $\mathbb R$, and $\mathbb C$.
We write $\widehat{\mathbb C}:=\mathbb C\cup\{\infty\}$ for the Riemann sphere.
For a set $A$, we denote by $\mathcal{P}(A)$ its power set.
For a topological space $A$, we denote by $\mathcal B(A)$ its Borel $\sigma$-algebra.
A set is called countable if it is finite or countably infinite; in
particular, every finite set is countable.

Whenever possible, we follow the notation and terminology of
\cite{MR2747724,MR4002398} in order to facilitate comparison with the existing literature.

\section{Formulation of the Julia--Fatou Theory for Random Systems with Complete Connections}
\label{sec:formulation}
\subsection{%
RSCCs and Admissible Composition Families}

In this subsection, we introduce random systems with complete connections
(RSCCs), admissible index words and reachable states, and the associated
families of admissible compositions on a compact metric space.

 \begin{definition}[{\cite[Definition~1.1.1]{MR1070097}}]\label{def:RSCC}
     A random system with complete connections (RSCC) is a quadruple $\{ (W, \mathcal{W}), (X,\mathcal{X}),u ,P \} $, where
     \begin{enumerate}[label=\textup{(\roman*)}]
         \item $(W, \mathcal{W})$ and $(X,\mathcal{X})$ are arbitrary measurable spaces;
         \item $u \colon W \times X \to W $ is a $(\mathcal{W} \otimes \mathcal{X}, \mathcal{W})$-measurable map;
         \item $P$ is a transition probability function from $(W,\mathcal{W})$ to $(X,\mathcal{X})$, that is, for each $w \in W$, the map $A \mapsto P(w,A)$ defines a probability 
measure on $(X,\mathcal{X})$, and for each $A \in \mathcal{X}$, 
the map $w \mapsto P(w,A)$ is $\mathcal{W}$-measurable.
     \end{enumerate}
 \end{definition}
 
In the above definition, we call $W$ the \emph{state space}, $X$ the \emph{index space}, and $u$ the \emph{update map}. 
We emphasize that, in general, no restrictions are imposed on the cardinality,
metric structure, or topological structure of the state space $W$ or the index
space $X$ in the definition of an RSCC.
Accordingly, no continuity assumptions are required for the update map $u$
or for the transition probability function $P$.
For concrete examples illustrating this generality, we refer the reader to
\cite[Section~1.2]{MR1070097}.

\begin{notation}
We write $x^{(n)}=(x_{1},\ldots,x_{n})\in X^{n}$.
For $n\in\mathbb{N}$, define inductively the maps
\[
u^{(n)} \colon W \times X^{n} \to W
\]
by
\[
u^{(n)}(w,x^{(n)})
=
\begin{cases}
u(w,x_1), & \text{if } n=1,\\[6pt]
u\big(u^{(n-1)}(w,x^{(n-1)}),x_n\big),
& \text{if } n\ge2.
\end{cases}
\]
For simplicity, we write $u^{(n)}(w, x^{(n)})$ as $w x^{(n)}$ whenever no confusion arises.
For $r\in\mathbb{N}$, we define the $r$-step transition probability
$P_r$ from $(W,\mathcal W)$ to $(X^r,\mathcal X^r)$ by
\[
P_r(w,A)
=
\int_{X^r}
\mathbf{1}_A(x^{(r)})
\, P(w,dx_1)
P(w x_1,dx_2)
\cdots
P(w x^{(r-1)},dx_r),
\]
for $w\in W$ and $A\in\mathcal X^r$,
with the convention that $P_1(w,A)=P(w,A)$.

Moreover, for $n,r\in\mathbb{N}$, $w\in W$ and $A\in\mathcal X^r$, we define
\[
P_r^{\,n}(w,A)
:=
P_{n+r-1}\big(w, X^{n-1}\times A\big).
\]
\end{notation}

The following existence theorem provides a cornerstone of the framework
developed in this paper.

\begin{theorem}[{\cite[Theorem~1.1.2]{MR1070097}}]\label{thm:RSCC-existence}
Let $\{(W,\mathcal{W}), (X,\mathcal{X}), u, P\}$
be an RSCC,
and fix an arbitrary state $w_0\in W$.
Then there exists a unique probability measure $\mathbf{P}_{w_{0}}$ on 
$(X^{\mathbb{N}}, \mathcal{X}^{\mathbb{N}})$ 
and a sequence of $X$-valued random variables $(\xi_{n})_{n\in\mathbb{N}}$ defined on 
$(X^{\mathbb{N}}, \mathcal{X}^{\mathbb{N}}, \mathbf{P}_{w_{0}})$ such that, 
for all $m,n,r\in\mathbb{N}$ and $A\in\mathcal{X}^{r}$, the following hold:
\begin{enumerate}[label=\textup{(\roman*)}]
    \item $\mathbf{P}_{w_{0}}\big(
    (\xi_{n},\ldots,\xi_{n+r-1})
    \in A\big)=P_{r}^{n}(w_{0},A)$;

    \item $\mathbf{P}_{w_{0}}\big(
    (\xi_{n+m},\ldots,\xi_{n+m+r-1})
    \in A \mid \xi^{(n)}\big)
    =P_{r}^{m}(w_{0}\xi^{(n)},A)$, $\mathbf{P}_{w_{0}}$-a.s.;

    \item $\mathbf{P}_{w_{0}}\big(
    (\xi_{n+m},\ldots,\xi_{n+m+r-1})
    \in A \mid \xi^{(n)},\zeta^{(n)}\big)
    =P_{r}^{m}(\zeta_{n},A)$, $\mathbf{P}_{w_{0}}$-a.s.;
\end{enumerate}
where $\xi^{(n)}=(\xi_{1},\ldots,\xi_{n})$, $\zeta_{n}=w_{0}\xi^{(n)}$, and $\zeta^{(n)}=(\zeta_{1},\ldots,\zeta_{n})$.

Moreover, the sequence
$(\zeta_n)_{n\in\mathbb N_0}$
with $\zeta_0=w_0$ forms a $W$-valued homogeneous Markov chain whose
transition operator
\begin{equation*}
    \U f(w) = \int_{X} f(wx)\,P(w,dx) , 
    \qquad f\in B_b(W,\mathcal W),
\end{equation*}
acts on the Banach space $B_b(W,\mathcal W)$ 
of all bounded $\mathcal{W}$-measurable complex-valued functions on $W$.
\end{theorem}

\begin{assumption}\label{ass:CDIS}
Throughout the remainder of this paper, the index space $X$ is
at most countable and is endowed with the discrete topology,
and $\mathcal X=\mathcal P(X)$.
\end{assumption}

We introduce the notions of admissible index words and reachable states
associated with an RSCC.

\begin{definition}[Admissible words and reachable states]
\label{def:admissible-words}
Let $\{(W,\mathcal{W}), (X,\mathcal{X}), u, P\}$ be an RSCC.
Fix $w \in W$.
For $n\in\mathbb{N}$ and $a^{(n)}=(a_1,\ldots,a_n)\in X^n$,
we denote by
\[
[a^{(n)}]=[a_1,\ldots,a_n]
:=
\bigl\{
(x_j)_{j\in\mathbb{N}}\in X^{\mathbb{N}}:
x_j=a_j \text{ for } j=1,\ldots,n
\bigr\}
\]
the cylinder set determined by $a^{(n)}$.
\begin{enumerate}[label=\textup{(\roman*)}]
\item
The set of all \emph{admissible index words of length $n$ from the state $w$} is defined by
\[
X_{w,n}
:=
\bigl\{
  x^{(n)} \in X^{n} \ :\ 
  \mathbf{P}_{w}\big([x^{(n)}]\big) > 0
\bigr\}.
\]

\item
The set of all \emph{admissible index words of finite length from the state $w$} is defined by
\[
X_{w,*}
:=
\bigcup_{n \in \mathbb{N}} X_{w,n}.
\]

\item
The set of \emph{reachable states from $w$} is defined by
\[
\operatorname{Reach}(w)
:=
\bigl\{
  w x^{(n)} \in W \ : \ 
  x^{(n)}\in X_{w,*}
\bigr\}.
\]
\end{enumerate}
\end{definition}

\begin{notation}
Let $(Y,d_Y)$ be a nonempty compact metric space.
We denote by $\mathrm{CM}(Y)$ the space of all continuous self-maps on $Y$, 
equipped with the compact--open topology.

For a finite sequence $\gamma^{(n)} = (\gamma_{1},\ldots,\gamma_{n}) 
\in \mathrm{CM}(Y)^{n}$ and integers $1 \le M \le N \le n$, we define
\[
\gamma_{N,M} := \gamma_{N} \circ \cdots \circ \gamma_{M}.
\]

For an infinite sequence $\gamma = (\gamma_{k})_{k\in\mathbb{N}} 
\in \mathrm{CM}(Y)^{\mathbb{N}}$ and integers $1 \le M \le N$, we define
\[
\gamma_{N,M} := \gamma_{N} \circ \cdots \circ \gamma_{M}.
\]
\end{notation}

The following definition introduces the central class of dynamical systems
considered in this paper.

\begin{definition}\label{def:basic-setup}
Let $\{(W,\mathcal{W}), (X,\mathcal{X}), u, P\}$ be an RSCC and let
$\{\Gamma_x\}_{x\in X}$ be a family of nonempty subsets of $\mathrm{CM}(Y)$.
We call
\[
S := \{(W,\mathcal{W}), (X,\mathcal{X}), u, P, \{\Gamma_x\}_{x\in X}\}
\]
a \emph{random system of continuous maps acting on $Y$ with complete connections},
or simply an \emph{RSCC on $Y$}.
\end{definition}

\medskip

Throughout this section, we fix a nonempty compact metric space $(Y,d_Y)$ and an RSCC on $Y$  given by
\[
S := \{(W,\mathcal{W}), (X,\mathcal{X}), u, P, \{\Gamma_x\}_{x\in X}\}
\]
under Assumption~\ref{ass:CDIS}.
All statements below are understood with respect to this fixed RSCC,
unless explicitly stated otherwise.

\begin{definition}[Admissible composition families]
\label{def:admissible-composition-family}
For each $w \in W$, we define the set of
\emph{admissible finite compositions from $w$} by
\[
H_w(S)
:=
\left\{
\gamma_n \circ \cdots \circ \gamma_1 \in \mathrm{CM}(Y)
\; : \;
\substack{
n \in \mathbb{N},\\
x^{(n)}=(x_1,\ldots,x_n) \in X_{w,n},\\
\gamma_j \in \Gamma_{x_j}\ (j=1,\ldots,n)
}
\right\}.
\]
Moreover, for $v\in \operatorname{Reach}(w)$, we define
\[
H_w^v(S)
:=
\left\{
\gamma_n \circ \cdots \circ \gamma_1 \in \mathrm{CM}(Y)
\; : \;
\substack{
n \in \mathbb{N},\\
x^{(n)}=(x_1,\ldots,x_n) \in X_{w,n},\\
\gamma_j \in \Gamma_{x_j}\ (j=1,\ldots,n),\\
w x^{(n)} = v
}
\right\}.
\]
For $v\notin \operatorname{Reach}(w)$, we define $H_w^v(S):=\emptyset$.
\end{definition}

\subsection{Julia, Fatou and Kernel Julia Sets}
\label{sec:julia-fatou-kernel}

Using the admissible composition families introduced above, we now define
the Julia, Fatou, and kernel Julia sets at each state and establish their
basic invariance properties. We also explain how the GDMS and MRDS
framework of \cite{MR4002398} is recovered as a special case.

\begin{definition}[Julia and Fatou sets at a state]
\label{def:Julia-RSCC}
For each state $w \in W$, we define the \emph{Julia set at $w$} by
\[
J_{w}(S)
:= 
\bigl\{
  y \in Y \; :\;
  H_{w}(S) \text{ is not equicontinuous on any neighborhood of } y
\bigr\}.
\]
The corresponding \emph{Fatou set} is defined as the complement $F_{w}(S)=Y \setminus J_{w}(S)$.
\end{definition}

The following lemma follows immediately from the definition.

\begin{lemma}\label{lem:Fatou-Julia-open-compact}
For each $w \in W$, the Fatou set $F_w(S)$ is open in $Y$, 
and the Julia set $J_w(S)$ is compact in $Y$.
\end{lemma}

Next, we define the notion of invariance in the setting of RSCCs.
Recall that we write $wx := u(w,x)\in W$ for $w \in W$ and $x \in X$.

\begin{definition}[Forward and backward invariance]
For each $w \in W$, let $E_w$ be a subset of $Y$.
\begin{enumerate}[label=\textup{(\roman*)}]
\item
The family $(E_w)_{w\in W}$ is said to be
\emph{forward $S$-invariant} if for every $w\in W$ and every $x\in X_{w,1}$,
\[
\bigcup_{f\in\Gamma_x} f(E_w) \subset E_{wx}.
\]

\item
The family $(E_w)_{w\in W}$ is said to be
\emph{backward $S$-invariant} if for every $w\in W$ and every $x\in X_{w,1}$,
\[
\bigcup_{f\in\Gamma_x} f^{-1}(E_{wx}) \subset E_w.
\]
\end{enumerate}
\end{definition}

By a standard argument based on the stability of equicontinuity under composition and the openness of the maps in $\Gamma_x$, we obtain the following invariance property.

\begin{lemma}[Invariance of Julia--Fatou sets]
\label{lem:Rscc-Fatou-Julia-invariance}
Let $\mathrm{OCM}(Y)$ denote the set of open continuous self-maps of $Y$.
Assume that $\Gamma_x\subset \mathrm{OCM}(Y)$ for every $x\in X$.
Then the family of Fatou sets $(F_w(S))_{w\in W}$ is forward $S$-invariant.
Consequently, the family of Julia sets $(J_w(S))_{w\in W}$ is backward
$S$-invariant.
\end{lemma}
We next introduce the kernel Julia set for RSCCs as a natural generalization
of \cite[Definition~2.7]{MR2747724} and \cite[Definition~2.17]{MR4002398}.

\begin{definition}[Kernel Julia set at a state]
\label{def:kernel-julia-set}
For each $w\in W$, we define the \emph{kernel Julia set at $w$} by
\[
J_{\ker,w}(S)
:=
\bigcap_{v\in \operatorname{Reach}(w)}
\ \bigcap_{\gamma\in H_w^v(S)}
\gamma^{-1}\!\bigl(J_v(S)\bigr).
\]
\end{definition}

\begin{lemma}[Forward invariance of kernel Julia sets]
\label{lem:kernel-Julia-forward-invariance}
The family $(J_{\ker,w}(S))_{w\in W}$ is forward $S$-invariant.
\end{lemma}

\begin{proof}
Fix $w\in W$, $x\in X_{w,1}$, and $f\in\Gamma_x$, and let $z\in J_{\ker,w}(S)$.
Set $y:=f(z)$. We show that $y\in J_{\ker,wx}(S)$.
Let $u\in W$ and $\gamma\in H_{wx}^u(S)$ be arbitrary.
Then $\gamma\circ f\in H_w^u(S)$.
Since $z\in J_{\ker,w}(S)$, we have
\[
(\gamma\circ f)(z)\in J_u(S),
\]
and hence
\[
\gamma(y)=\gamma(f(z))\in J_u(S).
\]
As $u$ and $\gamma$ were arbitrary, it follows that $y\in J_{\ker,wx}(S)$.
Therefore,
\[
f(J_{\ker,w}(S))\subset J_{\ker,wx}(S).
\]
Since $w\in W$, $x\in X_{w,1}$, and $f\in\Gamma_x$ were arbitrary, we obtain
\[
\bigcup_{f\in\Gamma_x} f(J_{\ker,w}(S)) \subset J_{\ker,wx}(S)
\qquad
\text{for all } w\in W \text{ and } x\in X_{w,1}.
\]
This is exactly the forward $S$-invariance of $(J_{\ker,w}(S))_{w\in W}$.
\end{proof}

We now show that each GDMS considered in \cite{MR4002398}, together
with its associated Markov random dynamical system (MRDS), can be realized
as an RSCC.
The i.i.d.\ setting of \cite{MR2747724} corresponds to a single
vertex and a single edge.
Let $m\in\mathbb{N}$ and let $(\tau_{ij})_{i,j=1}^{m}$ be a family of Borel measures
on $\mathrm{CM}(Y)$ satisfying
\[
\sum_{j=1}^{m} \tau_{ij}(\mathrm{CM}(Y)) = 1,
\qquad i=1,\ldots,m.
\]
In \cite{MR4002398},
this family of measures
induces a Markov chain on
$Y\times\{1,\ldots,m\}$ whose transition probability from
$(y,i)\in Y\times\{1,\ldots,m\}$ to a set $B\times\{j\}\subset Y\times\{1,\ldots,m\}$
is given by
\[
\tau_{ij}\bigl(\{f\in \mathrm{CM}(Y): f(y)\in B\}\bigr).
\]
The vertex set is defined by $V:=\{1,\ldots,m\}$ and the edge set by
\[
E:=\{(i,j)\in V\times V:\ \tau_{ij}(\mathrm{CM}(Y))>0\}.
\]
For each edge $e=(i,j)\in E$, one sets
\(
\Gamma_{e}:=\supp(\tau_{ij}).
\)
The triple $(V,E,\{\Gamma_{e}\}_{e\in E})$ then defines a GDMS in the sense of
\cite{MR4002398}, and the above Markov chain on $Y\times V$ is the associated
MRDS.

\begin{example}
\label{ex:rscc-gdms-mrds}
We show that the GDMS recalled above, together with its associated
MRDS, can be realized as an RSCC.
We set
\[
W:=V,\qquad \mathcal{W}:=\mathcal{P}(W),
\qquad
X:=E,\qquad \mathcal{X}:=\mathcal{P}(X).
\]
The update map $u:W\times X\to W$ is defined by
\[
u(w,x):=
\begin{cases}
t(x) & \text{if } i(x)=w,\\
w & \text{otherwise},
\end{cases}
\]
where $i(x)$ and $t(x)$ denote the initial and terminal vertices of the edge
$x\in E$.
The transition probability function $P:W\times X\to[0,1]$ is given by
\[
P(w,x):=
\begin{cases}
\tau_{ij}(\mathrm{CM}(Y))
& \text{if } w=i=i(x)\text{ and } j=t(x),\\
0 & \text{otherwise}.
\end{cases}
\]

With these choices, the RSCC
\(
\{(W,\mathcal{W}),(X,\mathcal{X}),u,P\}
\)
generates the same admissible index sequences as the Markov chain underlying
the MRDS of \cite{MR4002398}, and the associated family
$\{\Gamma_x\}_{x\in X}$ coincides with the collection of edge maps of the GDMS.

In particular, for each vertex $i\in V$ identified with a state $w\in W$, the
family of admissible compositions $H_i(S)$ defined in
\cite[Definition~2.3]{MR4002398} coincides with the family $H_w(S)$ introduced in
Definition~\ref{def:admissible-composition-family}.
Consequently, the Julia set $J_i(S)$ in
\cite[Definition~2.4]{MR4002398} agrees with the Julia set $J_w(S)$ defined in
Definition~\ref{def:Julia-RSCC} under this identification.
\end{example}

Conversely, it is also clear from the above example that any RSCC with finite
state space and finite index set reduces to the GDMS and MRDS framework
considered in \cite{MR4002398}.

\subsection{Propagation of Kernel Julia Set Emptiness and the Jump Phenomenon}
\label{sec:propagation}
In this subsection, we study the behavior of the kernel Julia sets $J_{\ker,w}(S_\tau)$ as the state $w\in W$ varies.
In particular, we investigate how the emptiness of $J_{\ker,w}(S_\tau)$
propagates when the state is changed.

We first introduce and formalize the \emph{jump phenomenon}.
This phenomenon is a genuinely new feature of the present framework and does
not occur in the setting of
graph-directed Markov systems with finitely many edges and finitely many vertices studied in
the previous literature \cite{MR4002398}.

\begin{definition}[Emptiness jump of kernel Julia sets]
\label{def:kernel-julia-jump}
In addition to Assumption~\ref{ass:CDIS}, suppose that the state space $W$ is a
metric space and that the $\sigma$-algebra $\mathcal W$ coincides with its
Borel $\sigma$-algebra $\mathcal B(W)$.
Let $d_W$ denote the metric on $W$.
We say that the RSCC $S$ on $Y$ \emph{exhibits an emptiness jump in kernel Julia
sets}, or simply a \emph{jump}, if there exist a state $w_0\in W$, an admissible
sequence $(x_n)_{n\in\mathbb{N}}\in \operatorname{supp}(\mathbf P_{w_0})$, and a
state $w_\infty\in W$ such that the associated state sequence
\[
w_n := w_0 x^{(n)}, \qquad n\in\mathbb{N},
\]
satisfies the following properties:
\begin{enumerate}[label=\textup{(\roman*)}, leftmargin=2.2em]
\item
$d_W(w_n,w_\infty)\to 0$ as $n\to\infty$;
\item
$J_{\ker,w_n}(S)=\emptyset$ for all $n\in\mathbb{N}$;
\item
$J_{\ker,w_\infty}(S)\neq\emptyset$.
\end{enumerate}
\end{definition}

\begin{example}
\label{ex:kernel-julia-jump}
We construct a polynomial RSCC on the Riemann sphere
which exhibits an emptiness jump of kernel Julia sets.

\medskip
\noindent
\textbf{Step 1.} Construction of the RSCC.

Let $Y=\widehat{\mathbb C}$ be equipped with the spherical metric.
Define the state space
\[
W:=\{0\}\cup\{1/n:n\in\mathbb{N}\}\cup\{2\}\subset\mathbb R
\]
endowed with the Euclidean metric.
Then $(W,d_W)$ is a compact metric space and $\mathcal W=\mathcal B(W)$.
Let
\[
X:=\{x_1,x_2\}, \qquad \mathcal X=\mathcal{P}(X).
\]

Define the update map $u:W\times X\to W$ by
\[
u(1/n,x_1)=1/(n+1),\qquad
u(1/n,x_2)=2,
\]
\[
u(0,x_1)=0,\qquad u(0,x_2)=2,
\]
\[
u(2,x_1)=2,\qquad u(2,x_2)=2.
\]

The transition probabilities are specified by
\[
P(1/n,\{x_1\})=1-2^{-\frac1n},\qquad
P(1/n,\{x_2\})=2^{-\frac1n}.
\]
\[
P(0,\{x_1\})=1,\quad P(0,\{x_2\})=0,
\]
\[
P(2,\{x_1\})=0,\quad P(2,\{x_2\})=1.
\]

Finally, define the polynomial maps
\[
f(z)=z^2,\qquad \Gamma_{x_1}=\{f\},
\]
\[
g(z)=\frac{z^2}{2},\qquad \Gamma_{x_2}=\{f,g\}.
\]

We denote the resulting RSCC by
\[
S=\{(W,\mathcal W),(X,\mathcal X),u,P,\{\Gamma_x\}_{x\in X}\}.
\]

\medskip
\noindent
\textbf{Step 2.} Structure of the Julia set at the state $2$.

At the state $2$, one has $P(2,\{x_2\})=1$,
and hence the admissible family coincides with the polynomial semigroup
\[
H_2(S)=\langle f,g\rangle 
:=
\bigl\{
h_n\circ\cdots\circ h_1:
n\in\mathbb N,\ h_j\in\{f,g\}\ (1\le j\le n)
\bigr\}.
\]

\medskip
\noindent
\emph{Claim A.
\(
J_2(S)
=
\{\,z\in\mathbb C:1\le |z|\le 2\,\}.
\)
}

\smallskip
\noindent
\emph{Proof.}
At the state $2$, the admissible family is the semigroup
$H_2(S)=\langle z^2, z^2/2\rangle$.
This semigroup coincides with that considered in
\cite[Example~1]{MR1397693}.
It is shown therein that its Julia set is precisely the closed annulus
\(
\{\,z\in\mathbb C:1\le |z|\le 2\,\}.
\)

\medskip
\noindent
\emph{Claim B.
\(
J_{\ker,2}(S)=\emptyset.
\)
}

\smallskip
\noindent
\emph{Proof.}
We compute
\[
f^{-1}(J_2(S))
=
\{\,z\in\mathbb C:1\le |z|\le \sqrt2\},
\qquad
g^{-1}(J_2(S))
=
\{\,z\in\mathbb C:\sqrt2\le |z|\le 2\}.
\]
Hence
\[
f^{-1}(J_2(S))
\cap
g^{-1}(J_2(S))
=
\{\,z\in\mathbb C:|z|=\sqrt2\}.
\]
Furthermore,
\[
(f^2)^{-1}(J_2(S))
=
\{\,z\in\mathbb C:1\le |z|\le 2^{1/4}\}.
\]
Since $2^{1/4}<\sqrt2$, it follows that
\[
\bigcap_{h\in H_2(S)} h^{-1}(J_2(S))
=
\emptyset.
\]
Therefore $J_{\ker,2}(S)=\emptyset$.

\medskip
\noindent
\textbf{Step 3.} Vanishing of the kernel Julia sets along $(w_n)_{n\in \mathbb{N}}$.

Let $w_0=1$ and consider the admissible index sequence
$(x_n)=(x_1,x_1,\dots)$.
Then
\[
w_n=\frac{1}{n+1}\to 0.
\]

\medskip
\noindent
\emph{Claim C.
For every $n\in\mathbb{N}$, one has
\(
J_{\ker,w_n}(S)=\emptyset.
\)
}

\smallskip
\noindent
\emph{Proof.}
Since $P(w_n,\{x_2\})>0$, $u(w_n,x_2)=2$, and
$f\in\Gamma_{x_2}$, Lemma~\ref{lem:kernel-Julia-forward-invariance}
and Claim B yield
\[
f\bigl(J_{\ker,w_n}(S)\bigr)
\subset J_{\ker,2}(S)=\emptyset.
\]
Therefore $J_{\ker,w_n}(S)=\emptyset$.

\medskip
\noindent
\textbf{Step 4.} Non-vanishing at the limit state.

At the limit state $w_\infty=0$, the index $x_2$ is no longer admissible.
The system reduces to deterministic iteration of $f(z)=z^2$.
Consequently,
\[
J_{\ker,0}(S)
=
J(f)
=
\{\,z\in \mathbb C:|z|=1\,\}
\neq\emptyset.
\]

\medskip
\noindent
We conclude that $S$ exhibits an emptiness jump in kernel Julia sets.
\end{example}

\begin{remark}
\label{rem:jump-arbitrary-state-space}
By an argument analogous to that in
Example~\ref{ex:kernel-julia-jump}, for every metric space $W$
having an accumulation point, one can choose $(X, \mathcal X )$, $u$, $P$, and
$\{\Gamma_x\}_{x\in X}$ so that the resulting RSCC on
$\widehat{\mathbb C}$ exhibits an emptiness jump.
Indeed, choose an accumulation point $w_\infty\in W$, an auxiliary
state $v_*\ne w_\infty$, and a sequence of distinct states
$(w_n)_{n\in\mathbb N}$ in $W\setminus\{v_*\}$ converging to
$w_\infty$.
With $X=\{x_1,x_2\}$, set
\[
u(w_n,x_1)=w_{n+1},
\qquad
u(w_n,x_2)=v_*,
\]
and let $w_\infty$ be fixed under its only admissible index $x_1$,
while $v_*$ is fixed under its only admissible index $x_2$.
Give both indices positive probability at each $w_n$, and take
$\Gamma_{x_1}=\{f\}$ and $\Gamma_{x_2}=\{f,g\}$ as in
Example~\ref{ex:kernel-julia-jump}.
On the remaining states, the update map and transition probabilities
may be defined by deterministic loops.
Since the distinguished states form a countable Borel subset of $W$,
these extensions can be chosen Borel measurable.
The proof of Example~\ref{ex:kernel-julia-jump} then applies without
change.
\end{remark}

An emptiness jump implies that $J_{\ker,w}(S)$ is nonempty
for some state $w$, so the kernel-emptiness assumption of
Theorem~\ref{thm:CP-Fatou} is not satisfied.
We next give sufficient conditions excluding such jumps.

\begin{proposition}[No jump for discrete state spaces]
\label{prop:no-jump-discrete-W}
Let $S$ be an RSCC on $Y$.
Assume that the state space $W$ is endowed with the discrete metric $d_W$
given by
\[
d_W(w,v)
:=
\begin{cases}
0 & \text{if } w=v,\\
1 & \text{if } w\ne v,
\end{cases}
\qquad w,v\in W,
\]
and that $\mathcal W$ coincides with the Borel $\sigma$-algebra $\mathcal B(W)$.
Then $S$ does not exhibit an emptiness jump in kernel Julia sets in the sense
of Definition~\ref{def:kernel-julia-jump}.
\end{proposition}
\begin{proof}
Assume, to the contrary, that $S$ exhibits a jump.
Then there exist a state $w_0\in W$, an admissible sequence
$(x_n)_{n\in\mathbb{N}}\in \operatorname{supp}(\mathbf P_{w_0})$, and a state
$w_\infty\in W$ such that the associated state sequence
\[
w_n := w_0 x^{(n)}, \qquad n\in\mathbb{N},
\]
satisfies conditions \textup{(i)}--\textup{(iii)} of
Definition~\ref{def:kernel-julia-jump}.
Since $d_W$ is the discrete metric, the convergence $d_W(w_n,w_\infty)\to 0$
implies the existence of an integer $N\in\mathbb{N}$ such that
$w_n=w_\infty$ for all $n\ge N$.
Indeed, the inequality $d_W(w_n,w_\infty)<1$ forces $w_n=w_\infty$.
Consequently,
\[
J_{\ker,w_\infty}(S)=J_{\ker,w_n}(S)=\emptyset
\quad\text{for all } n\ge N,
\]
which contradicts condition \textup{(iii)} of
Definition~\ref{def:kernel-julia-jump}.
This contradiction completes the proof.
\end{proof}

\begin{remark}
Proposition~\ref{prop:no-jump-discrete-W} shows that no jump phenomenon can
occur in the framework of
graph-directed
Markov systems with finitely many vertices and edges considered in
\cite{MR4002398}.
More generally, even when the number of vertices is allowed to be infinite,
the jump phenomenon is still ruled out as long as the state space is endowed
with the discrete topology.
In particular, the jump phenomenon can arise only when the state space
possesses accumulation points.
\end{remark}

We next introduce the notion of irreducibility.

\begin{definition}[{\cite[Definition~3.3.15]{MR1070097}}]
\label{def:phi-irreducible-RSCC}
Let $\varphi$ be a
$\sigma$-finite measure on $(W,\mathcal W)$.
For each $w\in W$, let $\mathbf P_w$ be the probability measure on
$(X^{\mathbb{N}},\mathcal X^{\mathbb{N}})$ associated with the initial state $w$
as in Theorem~\ref{thm:RSCC-existence}, and let
$(\zeta_n)_{n\in\mathbb{N}_0}$ be the corresponding $W$-valued Markov chain,
defined by
\[
\zeta_0:=w,
\qquad
\zeta_n:=w\xi^{(n)} \quad (n\in\mathbb{N}).
\]
For $A\in\mathcal W$ and $w\in W$, we define
\[
L(w,A)
:=
\mathbf P_w\Bigl(
\bigcup_{n\in\mathbb{N}}
\{\xi\in X^{\mathbb{N}} : \zeta_n\in A\}
\Bigr).
\]
The RSCC is said to be \emph{$\varphi$-irreducible} if
\[
L(w,A)>0
\quad
\text{for all } w\in W \text{ and all } A\in\mathcal W \text{ with } \varphi(A)>0.
\]
\end{definition}

\begin{remark}
\label{rem:irreducibility-and-stationary-distributions}
If the state space $W$ is at most countable, then $\varphi$-irreducibility reduces
to the usual notion of irreducibility for Markov chains by taking $\varphi$ to be
the counting measure on $W$.
\end{remark}

As a generalization of Lemma~2.15 in \cite{MR4002398}, the following result describes
the propagation of emptiness.

\begin{lemma}[Propagation of kernel-emptiness under $\varphi$-irreducibility]
\label{lemma:phi-irreducible-propagation-kernel-empty}
Let $\varphi$ be a $\sigma$-finite measure on $(W,\mathcal W)$ and assume that
$S$ is $\varphi$-irreducible.
Define
\[
A
:=
\{\, w\in W : J_{\ker,w}(S)=\emptyset \,\},
\]
and assume that $A\in\mathcal W$ and $\varphi(A)>0$.
Then $J_{\ker,w}(S)=\emptyset$ for all $w\in W$.
\end{lemma}

\begin{proof}
Let $w\in W$.
Since $S$ is $\varphi$-irreducible,
Definition~\ref{def:phi-irreducible-RSCC} implies that
$L(w,A)>0$.
Hence, there exist $n\in\mathbb{N}$ and a word $x^{(n)}\in X_{w,n}$ such that
\[
v:=wx^{(n)}\in A,
\]
and therefore $J_{\ker,v}(S)=\emptyset$.

Choose maps
$\gamma_j\in\Gamma_{x_j}$ for $1\le j\le n$.
Set $w_0:=w$ and $w_k:=wx^{(k)}$ for $k=1,\ldots,n$.
By repeated application of Lemma~\ref{lem:kernel-Julia-forward-invariance}, we obtain
\[
(\gamma_n\circ\cdots\circ\gamma_1)\bigl(J_{\ker,w}(S)\bigr)
\subset J_{\ker,wx^{(n)}}(S)
=
J_{\ker,v}(S)=\emptyset,
\]
which forces $J_{\ker,w}(S)=\emptyset$.
Since $w\in W$ was arbitrary, the conclusion follows.
\end{proof}

\begin{theorem}[Absence of jumps under $\varphi$-irreducibility]
\label{thm:no-jump-phi-irreducible}
In addition to Assumption~\ref{ass:CDIS}, suppose that the state space $W$ is a
metric space and that the $\sigma$-algebra $\mathcal W$ coincides with its
Borel $\sigma$-algebra $\mathcal B(W)$.
Let $\varphi$ be a $\sigma$-finite measure on $(W,\mathcal W)$ and assume that
$S$ is $\varphi$-irreducible.
Define
\[
A
:=
\{\, w\in W : J_{\ker,w}(S)=\emptyset \,\}.
\]
Assume that $A\in\mathcal W$ and that $\varphi(A)>0$.
Then $S$ does not exhibit an emptiness jump in kernel Julia sets in the sense of
Definition~\ref{def:kernel-julia-jump}.
\end{theorem}

\begin{proof}
Suppose, for the sake of contradiction, that $S$ exhibits an emptiness jump in
kernel Julia sets.
Then there exist a state $w_0\in W$, an admissible index sequence
$(x_n)_{n\in\mathbb{N}}\in \operatorname{supp}(\mathbf P_{w_0})$, and a state
$w_\infty\in W$ such that, writing $w_n:=w_0x^{(n)}$, we have
\[
d_W(w_n,w_\infty)\to 0 \quad \text{as } n\to\infty,
\]
\[
J_{\ker,w_n}(S)=\emptyset \quad \text{for all } n\in\mathbb{N},
\qquad
J_{\ker,w_\infty}(S)\neq\emptyset.
\]
In particular, $w_n\in A$ for all $n\in\mathbb{N}$, and hence $A$ is nonempty.
By assumption, we moreover have $\varphi(A)>0$.

By Lemma~\ref{lemma:phi-irreducible-propagation-kernel-empty}, the
$\varphi$-irreducibility of $S$ together with the condition $\varphi(A)>0$
implies that
\[
J_{\ker,w}(S)=\emptyset
\qquad\text{for all } w\in W.
\]
This contradicts the fact that $J_{\ker,w_\infty}(S)\neq\emptyset$.
Therefore, $S$ cannot exhibit an emptiness jump in kernel Julia sets.
\end{proof}

\subsection[Admissible Infinite Sequences and Associated Probability Measures]{%
Admissible Infinite Sequences and Associated Probability Measures}
Throughout this subsection, let
\[
S=\{(W,\mathcal W),(X,\mathcal X),u,P,\{\Gamma_x\}_{x\in X}\}
\]
be an RSCC on $Y$ satisfying Assumption~\ref{ass:CDIS}.
We study the dynamics of RSCCs along admissible infinite sequences.
We first introduce the space of admissible infinite paths starting from a
given state and define the associated Julia set along each path.
We then introduce probability measures governing the choice of maps
and construct the corresponding measures on the spaces of admissible
infinite paths.

\begin{definition}[Admissible infinite sequences]
\label{def:Xi-w}
For each $w\in W$, we define the set of all
\emph{admissible infinite sequences from the state $w$} by
\[
\Xi_{w}(S)
:=
\Bigl\{
  \xi = (\gamma_{n}, x_{n})_{n\in\mathbb{N}}
  \in (\mathrm{CM}(Y)\times X)^{\mathbb{N}}
  \ : \ 
  (x_{n})_{n\in\mathbb{N}} \in \supp(\mathbf{P}_{w}),\ 
  \gamma_{n} \in \Gamma_{x_{n}}
  \text{ for all } n\in\mathbb{N}
\Bigr\}.
\]
\end{definition}

\begin{definition}[
Julia set along a path]
\label{def:Julia-path}
For each infinite sequence
\[
\xi=(\gamma_n,x_n)_{n\in\mathbb N}
\in(\mathrm{CM}(Y)\times X)^{\mathbb N},
\]
we define the \emph{Julia set along $\xi$} by
\[
J_\xi
:=
\Bigl\{
y\in Y:
\{\gamma_{n,1}\}_{n\in\mathbb N}
\text{ is not equicontinuous on any neighborhood of }y
\Bigr\}.
\]
\end{definition}

Up to this point, we have worked in a setting where, for each index
$x \in X$, a set $\Gamma_x \subset \mathrm{CM}(Y)$ is prescribed.
In the following, we formalize a framework in which the selection of a
continuous map from $\Gamma_x$ is determined by a probability measure
$\tau_x$ supported on $\Gamma_x$.

\begin{definition}\label{def:RSCC-tau}
Let $\{(W,\mathcal{W}), (X,\mathcal{X}), u, P\}$ be an RSCC.
For each $x\in X$, let $\tau_x$ be a Borel probability measure on $\mathrm{CM}(Y)$,
and set $\Gamma_x := \supp(\tau_x)$.
We write $\tau := \{\tau_x\}_{x\in X}$ and define
\[
S_\tau
:=
\{(W,\mathcal{W}), (X,\mathcal{X}), u, P, \{\Gamma_x\}_{x\in X}\}.
\]
We call $S_\tau$ the \emph{RSCC on $Y$ associated with $\tau$}.
\end{definition}

The notation $S_\tau$ specifies the measures $\tau_x$, whereas $S$
specifies only their supports $\Gamma_x$.
We next define the associated probability measure on $\Xi_w(S_\tau)$.

\begin{definition}[Path-space measure induced by $\tau$]
\label{def:Pwtilde}
Fix $w\in W$.
For $n\in\mathbb{N}$, let $x^{(n)}=(x_1,\ldots,x_n)\in X_{w,n}$, and let
$A_k\in\mathcal{B}(\Gamma_{x_k})$ $(1\le k\le n)$, where
$\mathcal{B}(\Gamma_{x_k})$ denotes the Borel $\sigma$--algebra on $\Gamma_{x_k}$.
We define the corresponding cylinder set by
\[
C(A_1,\ldots,A_n; x^{(n)})
:=
\Bigl\{
(\gamma_j,b_j)_{j\in\mathbb{N}}\in\Xi_w(S_\tau):
b_k=x_k,\ \gamma_k\in A_k \ (1\le k\le n)
\Bigr\}.
\]
Let $\mathcal{C}_w$ denote the collection of all such cylinder sets.
We define a set function $\tilde{\tau}_{w}$ on $\mathcal{C}_w$ by
\[
\tilde{\tau}_{w}
\bigl(C(A_1,\ldots,A_n; x^{(n)})\bigr)
:=
\mathbf{P}_{w}\bigl([x_1,\ldots,x_n]\bigr)
\prod_{k=1}^{n}\tau_{x_k}(A_k),
\qquad n\in\mathbb{N}.
\]
Let $\mathcal{G}_w:=\sigma(\mathcal{C}_w)$ be the $\sigma$--algebra generated by
$\mathcal{C}_w$.

For $n\in\mathbb N$, define
\[
K_{w,n}\bigl(((g_j,b_j))_{j=1}^{n-1},B\bigr)
:=
\sum_{a\in X}P(wb^{(n-1)},\{a\})
\int_{\mathrm{CM}(Y)}
\mathbf 1_B(g,a)\,d\tau_a(g),
\]
for $((g_j,b_j))_{j=1}^{n-1}\in(\mathrm{CM}(Y)\times X)^{n-1}$
and $B\in\mathcal B(\mathrm{CM}(Y)\times X)$.
Here $b^{(n-1)}=(b_1,\ldots,b_{n-1})$;
for $n=1$, the first argument is the empty tuple and $wb^{(0)}:=w$.
The measurability of $u$ and $P$ and the countability of $X$
ensure that each $K_{w,n}$ is a probability kernel.
By the Ionescu--Tulcea extension theorem, these kernels determine
a unique probability measure on $(\mathrm{CM}(Y)\times X)^{\mathbb N}$
equipped with its product $\sigma$--algebra.
Its index marginal is $\mathbf P_w$, and it assigns probability one
to $g_j\in\Gamma_{b_j}$ for every $j\in\mathbb N$.
Here the supports have full measure because
$\mathrm{CM}(Y)$ and $X^{\mathbb N}$ are separable metrizable.
Thus $\Xi_w(S_\tau)$ is a measurable set of full measure.
Restriction to this set yields the unique probability measure on
$(\Xi_w(S_\tau),\mathcal G_w)$ with the prescribed cylinder values.
We denote this probability measure again by $\tilde{\tau}_w$.
\end{definition}

\section{Generalization of Cooperation Principle~I}
\label{sec:generalization-CP}

\subsection{%
Standing Assumptions and Markov Operators on the Product Space
}
To establish Cooperation Principle~I, we first impose additional
assumptions on the RSCC. Under these assumptions, we introduce the
product space $\mathbb{Y}=Y\times W$ and study Markov operators on
$C(\mathbb{Y})$ together with their adjoint action on the space of
Borel probability measures on $\mathbb{Y}$.

\begin{assumption}\label{ass:standing-cooperation}
In addition to Assumption~\ref{ass:CDIS}, we impose the following standing
assumptions throughout this section.
\begin{enumerate}[label=\textup{(\roman*)}, leftmargin=2.8em]

\item\label{ass:finite-index-space}
The index space $X$ is finite, the state space $W$ is a compact
metric space, and $\mathcal W=\mathcal B(W)$.

\item\label{ass:update-equicontinuity}
The family
\(
\bigl\{
W\ni v\longmapsto va^{(n)}\in W:
n\in\mathbb N,\ a^{(n)}\in X^n
\bigr\}
\)
is equicontinuous on $W$.
In particular, for each $x\in X$, the one-step update map
\(
W\ni v\longmapsto vx=u(v,x)\in W
\)
is continuous.

\item\label{ass:prefix-continuity}
For every $w\in W$, we have
\[
\lim_{v\to w}
\sup_{n\in\mathbb N}
\sum_{a^{(n)}\in X^n}
\left|
\mathbf P_v([a^{(n)}])
-
\mathbf P_w([a^{(n)}])
\right|
=0.
\]
Since $\mathbf P_v([x])=P(v,\{x\})$, this condition implies, in
particular, that the map
\(
W\ni v\longmapsto P(v,\{x\})\in[0,1]
\)
is continuous for every $x\in X$.
\end{enumerate}
Throughout this section, we denote by $d_W$ the metric on $W$.
\end{assumption}
Graph-directed Markov systems with finite vertex and edge
sets, as considered in \cite{MR4002398}, can be naturally reformulated as
RSCCs satisfying Assumption~\ref{ass:standing-cooperation};
see Example~\ref{ex:rscc-gdms-mrds}.
Indeed, when the finite state space $W$ in that example is equipped
with the discrete metric, Assumption~\ref{ass:standing-cooperation}
holds automatically.
Examples with an infinite state space satisfying
Assumption~\ref{ass:standing-cooperation} are given in
Examples~\ref{ex:linear-reinforcement-rscc}
and~\ref{ex:clipped-reinforcement-rscc}.

Under Assumption~\ref{ass:standing-cooperation}, we now introduce the
product space on which the averaged dynamics will be formulated.
We also introduce the pointwise and measure-theoretic notions of
equicontinuity for iterates of Markov operators that will be used in
the proof of Cooperation Principle~I.

Note that the vertex set $V$ in the GDMS setting of \cite{MR4002398} is replaced
here by the state space $W$ of the RSCC.
For consistency with \cite{MR4002398}, we use the same notation
$\mathbb{Y}$ for the associated product space.

\begin{notation}
\label{not:product-space-basic}
Let $(Y,d_Y)$ be a nonempty compact metric space and let $(W,d_W)$ be the state space from Assumption~\ref{ass:standing-cooperation}.
We set
$
\mathbb{Y}:=Y\times W,
$
and equip $\mathbb{Y}$ with the metric
\[
d_{\mathbb{Y}}\bigl((y_1,w_1),(y_2,w_2)\bigr)
:=
\max\{d_Y(y_1,y_2),d_W(w_1,w_2)\}.
\]
This metric induces the product topology on $\mathbb{Y}$.
Then $\mathbb{Y}$ is a compact
metric space.

We denote by $C(\mathbb{Y})$ the Banach space of all complex-valued continuous
functions $\phi:\mathbb{Y}\to\mathbb{C}$ equipped with the supremum norm
\[
\|\phi\|_{\infty}
:=\sup_{(y,w)\in\mathbb{Y}}|\phi(y,w)|.
\]
Since $\mathbb{Y}$ is a compact
metric space,
the normed space
$\bigl(C(\mathbb{Y}),\|\cdot\|_{\infty}\bigr)$ is separable.

We write $\mathfrak{M}_{1}(\mathbb{Y})$ for the set of all Borel probability measures
on $\mathbb{Y}$.
We endow $\mathfrak{M}_{1}(\mathbb{Y})$ with the weak-$*$ topology,
that is, the coarsest topology for which every map
\[
\mathfrak{M}_{1}(\mathbb{Y})\ni \mu
\longmapsto
\int_{\mathbb{Y}} \phi \, d\mu,
\qquad \phi\in C(\mathbb{Y}),
\]
is continuous.
Then $\mathfrak{M}_{1}(\mathbb{Y})$ is compact in the weak-$*$ topology,
and this topology is induced by the metric
\[
d_{\mathfrak{M}_{1}(\mathbb{Y})}(\mu_1,\mu_2)
:=
\sum_{j=1}^{\infty}
\frac{1}{2^j}\,
\frac{
\bigl|
\int_{\mathbb{Y}} \phi_j\, d\mu_1
-
\int_{\mathbb{Y}} \phi_j\, d\mu_2
\bigr|
}{
1 +
\bigl|
\int_{\mathbb{Y}} \phi_j\, d\mu_1
-
\int_{\mathbb{Y}} \phi_j\, d\mu_2
\bigr|
},
\qquad \mu_1,\mu_2\in \mathfrak{M}_{1}(\mathbb{Y}),
\]
where $(\phi_j)_{j\in\mathbb{N}}$ is any sequence that is dense in $C(\mathbb{Y})$.

This metric depends only on the difference $\mu_1-\mu_2$.
In particular,
\[
d_{\mathfrak{M}_{1}(\mathbb{Y})}(\mu_1+\nu,\mu_2+\nu)
=
d_{\mathfrak{M}_{1}(\mathbb{Y})}(\mu_1,\mu_2)
\]
for every finite signed Borel measure $\nu$ such that
$\mu_1+\nu,\mu_2+\nu\in\mathfrak{M}_{1}(\mathbb{Y})$.
Moreover, by compactness, the uniform structure induced by this metric
coincides with the restriction of the canonical weak-$*$ uniform structure
on $C(\mathbb{Y})^*$.
Thus equicontinuity on $\mathfrak{M}_{1}(\mathbb{Y})$ is independent
of the choice of the dense sequence $(\phi_j)_{j\in\mathbb{N}}$.
\end{notation}
\begin{definition}[Markov operators]
\label{def:markov}
A linear operator \(M:C(\mathbb{Y})\to C(\mathbb{Y})\) is called a \emph{Markov operator} if \(M\mathbf{1}_{\mathbb{Y}}=\mathbf{1}_{\mathbb{Y}}\) and \(M\phi\ge 0\) for all \(\phi\in C(\mathbb{Y})\) with \(\phi\ge 0\), where \(\psi\ge 0\) means \(\psi(y,w)\ge 0\) for every \((y,w)\in\mathbb{Y}\).
\end{definition}

By the same argument as in \cite[Lemma~2.35]{MR4002398}, the following lemma
holds.
Note that the special case where $W$ is a finite discrete space corresponds precisely
to the setting of \cite[Lemma~2.35]{MR4002398}.

\begin{lemma}\label{lem:Markov-operator-norm}
Let \(M:C(\mathbb{Y})\to C(\mathbb{Y})\) be a Markov operator.  
Then its operator norm satisfies \(\|M\|=1\).  
Consequently, the adjoint operator \(M^{*}:C(\mathbb{Y})^{*}\to C(\mathbb{Y})^{*}\), defined by
\[
(M^{*}\mu)(\phi):=\mu(M\phi), 
\qquad 
\mu\in C(\mathbb{Y})^{*},\ \phi\in C(\mathbb{Y}),
\]
maps probability measures to probability measures; that is,
\(
M^{*}\bigl(\mathfrak{M}_{1}(\mathbb{Y})\bigr)\subset \mathfrak{M}_{1}(\mathbb{Y}).
\)
\end{lemma}

Following \cite[Definition~2.37]{MR4002398}, we introduce the following definition.

\begin{definition}\label{def:Fatou-meas}
Let $M:C(\mathbb{Y})\to C(\mathbb{Y})$ be a Markov operator, and let
\[
M^*:\mathfrak{M}_1(\mathbb{Y})
\longrightarrow
\mathfrak{M}_1(\mathbb{Y})
\]
be its adjoint.
We denote by
\[
\iota:\mathbb{Y}\to\mathfrak{M}_1(\mathbb{Y}),
\qquad
\iota(y,w):=\delta_{(y,w)},
\]
the natural embedding that sends each point to the corresponding
Dirac measure. We define the following sets:
\begin{enumerate}[label=\textup{(\roman*)}]
\item
$F_{\mathrm{meas}}(M^*)$ is the set of all
$\mu\in\mathfrak{M}_1(\mathbb{Y})$ for which there exists
a neighborhood $U$ of $\mu$ in $\mathfrak{M}_1(\mathbb{Y})$
such that the family $\{(M^*)^n\}_{n\in\mathbb N}$
is equicontinuous on $U$.

\item
$F_{\mathrm{pt}}^0(M^*)$ is the set of all
$(y,w)\in\mathbb{Y}$ at which the family
$\{(M^*)^n\circ\iota\}_{n\in\mathbb N}$
is equicontinuous.
\end{enumerate}
\end{definition}

The following lemma can be proved by the same argument as in \cite[Lemma~2.38]{MR4002398}.

\begin{lemma}[Pointwise equicontinuity via test functions]
\label{lem:pt-equicontinuity-M}
Let $M:C(\mathbb{Y})\to C(\mathbb{Y})$ be a Markov operator and 
$M^{*}:\mathfrak{M}_{1}(\mathbb{Y})\to\mathfrak{M}_{1}(\mathbb{Y})$ its adjoint.
For $(y,w)\in\mathbb{Y}$ the following are equivalent:
\begin{enumerate}[label=\textup{(\roman*)}]
    \item $(y,w)\in F_{\mathrm{pt}}^{0}(M^{*})$;
    \item For every $\phi\in C(\mathbb{Y})$, the family
$\{M^{n}\phi\}_{n\in\mathbb{N}_0}$
is equicontinuous at $(y,w)$.
\end{enumerate}
\end{lemma}

By the same argument as in \cite[Lemma~2.39]{MR4002398}, the following lemma holds.
We emphasize that the compactness of the state space is used essentially through
an application of the Arzel\`a--Ascoli theorem.

\begin{lemma}[Equivalence of measure-theoretic and pointwise equicontinuity]
\label{lem:Fmeas-Fpt-equivalence}
Let $M:C(\mathbb{Y})\to C(\mathbb{Y})$ be a Markov operator and 
$M^{*}:\mathfrak{M}_{1}(\mathbb{Y})\to\mathfrak{M}_{1}(\mathbb{Y})$ its adjoint.
Then the following are equivalent:
\begin{enumerate}[label=\textup{(\roman*)}]
    \item $F_{\mathrm{meas}}(M^{*})=\mathfrak{M}_{1}(\mathbb{Y})$;
    \item $F_{\mathrm{pt}}^{0}(M^{*})=\mathbb{Y}$.
\end{enumerate}
\end{lemma}

\subsection{Transition Operators and Their Iterations}
As a natural example of the Markov operators discussed in the preceding
subsection, we introduce the transition operator associated with an RSCC.
This operator describes an averaged dynamical system, in which the averaging
is taken both with respect to the probability of selecting each index in the
RSCC and with respect to the probability measure governing the choice of
continuous maps associated with that index.

\begin{notation}
We denote by $B_b(\mathbb{Y})$ the Banach space of all bounded
$\mathcal{B}(\mathbb{Y})$--measurable complex-valued functions on $\mathbb{Y}$,
equipped with the supremum norm
\[
\|\phi\|_{\infty}
:=
\sup_{(y,w)\in\mathbb{Y}} |\phi(y,w)|.
\]
\end{notation}

\begin{definition}[Transition operator]
\label{def:transition-operator}
Let 
\(
S_\tau := \{(W,\mathcal{W}), (X,\mathcal{X}), u, P, \{\Gamma_x\}_{x\in X}\}
\)
be an RSCC associated with $\tau:= \{\tau_x\}_{x\in X}$.
We define the \emph{transition operator} $M_\tau : B_b(\mathbb{Y}) \to B_b(\mathbb{Y})$
by
\[
M_{\tau}\phi(y,w)
:=
\sum_{x\in X_{w,1}}
P(w,\{x\})
\int_{\Gamma_x}
\phi\bigl(\gamma(y),\, wx\bigr)\, d\tau_x(\gamma),
\]
for $\phi\in B_b(\mathbb{Y})$ and $(y,w)\in\mathbb{Y}$.
\end{definition}

\begin{lemma}[The transition operator is Markov]
\label{lem:Mtau-is-Markov}
The transition operator $M_\tau$ maps $C(\mathbb{Y})$ into itself.
Moreover, the restriction
\[
M_\tau\colon C(\mathbb{Y})\to C(\mathbb{Y})
\]
is a Markov operator in the sense of Definition~\ref{def:markov}.
In particular, $\|M_\tau\|=1$ and
\[
M_\tau^{*}\bigl(\mathfrak{M}_{1}(\mathbb{Y})\bigr)\subset \mathfrak{M}_{1}(\mathbb{Y}).
\]
\end{lemma}

\begin{proof}
Let $\phi\in C(\mathbb{Y})$ and fix $x\in X$.
Define
\[
F_x(y,w)
:=
\int_{\Gamma_x}\phi\bigl(\gamma(y),\,wx\bigr)\,d\tau_x(\gamma),
\qquad (y,w)\in\mathbb{Y}.
\]
By Assumption~\ref{ass:standing-cooperation}\ref{ass:update-equicontinuity},
the map $w\mapsto wx$ is continuous.
Since $\phi$ is bounded and continuous and each $\gamma\in\Gamma_x$ is continuous
on $Y$, the map $(y,w,\gamma)\mapsto \phi(\gamma(y),wx)$ is continuous in $(y,w)$
and uniformly bounded by $\|\phi\|_\infty$.
Hence, if $(y_n,w_n)\to(y,w)$ in $\mathbb{Y}$, then
$\phi(\gamma(y_n),w_nx)\to\phi(\gamma(y),wx)$ for every $\gamma\in\Gamma_x$, and
the dominated convergence theorem yields $F_x(y_n,w_n)\to F_x(y,w)$.
Thus $F_x\in C(\mathbb{Y})$.
By definition,
\[
(M_\tau\phi)(y,w)
=
\sum_{x\in X_{w,1}} P(w,\{x\})\,F_x(y,w)
=
\sum_{x\in X} P(w,\{x\})\,F_x(y,w).
\]
Here the last equality holds because $P(w,\{x\})=0$ for
$x\notin X_{w,1}$.
By Assumption~\ref{ass:standing-cooperation}\ref{ass:prefix-continuity},
each function $w\mapsto P(w,\{x\})$ is continuous.
Since $X$ is finite by
Assumption~\ref{ass:standing-cooperation}\ref{ass:finite-index-space},
this is a finite sum of continuous functions, and therefore
$M_\tau\phi\in C(\mathbb{Y})$.
Hence $M_\tau$ maps $C(\mathbb{Y})$ into itself.

If $\phi\ge0$, then $F_x\ge0$ for all $x\in X$, which implies $M_\tau\phi\ge0$.
Moreover, since $\tau_x$ and $P(w,\cdot)$ are probability measures, we have
\[
(M_\tau\mathbf 1_{\mathbb{Y}})(y,w)
=
\sum_{x\in X_{w,1}} P(w,\{x\})\int_{\Gamma_x}1\,d\tau_x
=
\sum_{x\in X_{w,1}}P(w,\{x\})
=1.
\]
Thus $M_\tau:C(\mathbb{Y})\to C(\mathbb{Y})$ is a Markov operator.
The remaining assertions follow from Lemma~\ref{lem:Markov-operator-norm}.
\end{proof}
The iterated compositions of the transition operator satisfy the following
formula.

\begin{lemma}[Iteration formula]
\label{lem:Mn-formula}
Let $(y,w)\in\mathbb{Y}$, $n\in\mathbb{N}$, and $\phi\in B_b(\mathbb{Y})$.
Then the $n$-th iterate of the transition operator $M_\tau$ satisfies
\begin{align*}
(M_\tau^{\,n}\phi)(y,w)
&=
\sum_{x^{(n)}=(x_1,\ldots,x_n)\in X_{w,n}}
\mathbf{P}_w([x_1,\ldots,x_n]) \\
&\quad \cdot
\int_{\Gamma_{x_1}}\!\cdots\!\int_{\Gamma_{x_n}}
\phi\bigl(\gamma_{n,1}(y),\, w x^{(n)}\bigr)
\, d\tau_{x_n}(\gamma_n)\cdots d\tau_{x_1}(\gamma_1).
\end{align*}
\end{lemma}

\begin{proof}
We argue by induction on $n$.
For $n=1$, the formula is exactly the definition of $M_\tau$.

Assume that the formula holds for some $n\in\mathbb{N}$.
Let $(y,w)\in\mathbb{Y}$ and $\phi\in B_b(\mathbb{Y})$.
By definition of the transition operator and the induction hypothesis, we have
\begin{align*}
(M_\tau^{\,n+1}\phi)(y,w)
&=
\sum_{x_1\in X_{w,1}}
\mathbf P_w([x_1])
\int_{\Gamma_{x_1}}
(M_\tau^{\,n}\phi)\bigl(\gamma_1(y),\,wx_1\bigr)
\,d\tau_{x_1}(\gamma_1) \\
&=
\sum_{x_1\in X_{w,1}}
\sum_{(x_2,\ldots,x_{n+1})\in X_{wx_1,n}}
\mathbf P_w([x_1])\,\mathbf P_{wx_1}([x_2,\ldots,x_{n+1}]) \\
&\quad \cdot
\int_{\Gamma_{x_1}}\!\cdots\!\int_{\Gamma_{x_{n+1}}}
\phi\bigl(\gamma_{n+1,2}(\gamma_1(y)),\,wx_1\cdots x_{n+1}\bigr)
\, d\tau_{x_{n+1}}(\gamma_{n+1})\cdots d\tau_{x_1}(\gamma_1).
\end{align*}
The condition $(x_2,\ldots,x_{n+1})\in X_{wx_1,n}$ is equivalent to
$x^{(n+1)}=(x_1,\ldots,x_{n+1})\in X_{w,n+1}$.
Moreover, by the consistency of the RSCC transition probabilities,
\[
\mathbf P_w([x_1,\ldots,x_{n+1}])
=
\mathbf P_w([x_1])\,\mathbf P_{wx_1}([x_2,\ldots,x_{n+1}]),
\]
and by definition of the composed maps,
\(
\gamma_{n+1,2}(\gamma_1(y))=\gamma_{n+1,1}(y).
\)
Substituting these identities into the previous expression yields the desired
formula for $n+1$.
\end{proof}

Expressing the formula in the preceding lemma in terms of the
probability measure $\tilde{\tau}_w$ from Definition~\ref{def:Pwtilde}
yields the following integral representation.

\begin{lemma}[Integral representation]
\label{lem:Mn-integral-representation}
Let $(y,w)\in\mathbb{Y}$, $n\in\mathbb{N}$, and $\phi\in B_b(\mathbb{Y})$.
Then
\[
(M_\tau^{\,n}\phi)(y,w)
=
\int_{\Xi_w(S_{\tau})}
\phi\bigl(\gamma_{n,1}(y),\,w x^{(n)}\bigr)
\,d\tilde{\tau}_{w}(\xi),
\]
where $\xi=(\gamma_k,x_k)_{k\in\mathbb{N}}\in\Xi_w(S_{\tau})$ and
$x^{(n)}=(x_1,\ldots,x_n)$.
\end{lemma}

\begin{proof}
By Lemma~\ref{lem:Mn-formula}, $(M_\tau^{\,n}\phi)(y,w)$ is given by a finite sum
over $x^{(n)}\in X_{w,n}$ of iterated integrals with respect to the product
measures $\tau_{x_1}\otimes\cdots\otimes\tau_{x_n}$, weighted by
$\mathbf P_w([x_1,\ldots,x_n])$.
By Definition~\ref{def:Pwtilde}, the path-space measure $\tilde{\tau}_w$ assigns
to each cylinder set
$C(A_1,\ldots,A_n;x^{(n)})$ the value
\[
\mathbf P_w([x_1,\ldots,x_n])\prod_{k=1}^n\tau_{x_k}(A_k).
\]
Since the function
\(
\xi\mapsto \phi(\gamma_{n,1}(y),\,w x^{(n)})
\)
depends only on $(x_1,\ldots,x_n)$ and $(\gamma_1,\ldots,\gamma_n)$, its integral
with respect to $\tilde{\tau}_w$ coincides with the sum and iterated integrals in
Lemma~\ref{lem:Mn-formula}. This proves the claim.
\end{proof}

\subsection{Cooperation Principle~I}
After establishing several lemmas concerning the notions introduced so far,
we proceed to prove the Cooperation Principle~I, which constitutes the main
theorem of this paper.

As a generalization of \cite[Lemma~3.8]{MR4002398}, the following lemma holds.

\begin{lemma}
\label{lem:zero-probability-implies-Fpt}
Let $(y,w)\in\mathbb{Y}$.
Then
\(
\bigl\{
\xi=(\gamma_n,x_n)_{n\in\mathbb{N}}\in\Xi_w(S_\tau):y\in J_\xi
\bigr\}
\in\mathcal{G}_w.
\)
Assume that
\[
\tilde{\tau}_{w}\Bigl(
\bigl\{
\xi=(\gamma_n,x_n)_{n\in\mathbb{N}}\in\Xi_{w}(S_\tau)
:\ y\in J_{\xi}
\bigr\}
\Bigr)=0.
\]
Then
$(y,w)\in F_{\mathrm{pt}}^{0}(M_{\tau}^{*})$.
\end{lemma}

\begin{proof}
Fix $(y,w)\in\mathbb{Y}$, and write
$\overline{B}_Y(y,\rho):=\{z\in Y:d_Y(z,y)\le\rho\}$ for $\rho>0$.
For $m,r\in\mathbb{N}$ and $g\in\mathrm{CM}(Y)$, define
\[
T_{m,r}(g)
:=
\max_{\substack{z,z'\in\overline{B}_Y(y,1/m)\\d_Y(z,z')\le1/r}}
d_Y(g(z),g(z')).
\]
Since $Y$ is compact, the set
\[
\bigl\{(z,z')\in\overline{B}_Y(y,1/m)^2:
d_Y(z,z')\le1/r\bigr\}
\]
is compact and contains $(y,y)$.
Thus the maximum defining $T_{m,r}(g)$ exists.
The triangle inequality gives
\[
|T_{m,r}(g)-T_{m,r}(g')|
\le 2\sup_{z\in Y}d_Y(g(z),g'(z))
\qquad(g,g'\in\mathrm{CM}(Y)).
\]
Hence $T_{m,r}:\mathrm{CM}(Y)\to[0,\infty)$ is continuous
in the topology of uniform convergence, which coincides with
the compact--open topology.
Moreover, since continuous maps on the compact space $Y$ are
uniformly continuous, the map
\[
\mathrm{CM}(Y)^n\ni(g_1,\ldots,g_n)
\longmapsto g_n\circ\cdots\circ g_1\in\mathrm{CM}(Y)
\]
is continuous for each $n\in\mathbb{N}$, where the domain carries
the product topology.
Together with the $\mathcal G_w$-measurability of the coordinate maps
$\xi\mapsto\gamma_j$, this shows that
\[
\Xi_w(S_\tau)\ni\xi=(\gamma_j,x_j)_{j\in\mathbb{N}}
\longmapsto T_{m,r}(\gamma_{n,1})
\]
is $\mathcal G_w$-measurable for every $n,m,r\in\mathbb{N}$.

For $m\in\mathbb{N}$, we set
\[
\begin{aligned}
E_m(y,w):={}&\bigcap_{q\in\mathbb{N}}\bigcup_{r\in\mathbb{N}}
\bigcap_{n\in\mathbb{N}}
\Bigl\{\xi=(\gamma_j,x_j)_{j\in\mathbb{N}}\in\Xi_w(S_\tau):{}\\
&\max_{\substack{z,z'\in\overline{B}_Y(y,1/m)\\d_Y(z,z')\le1/r}}
d_Y(\gamma_{n,1}(z),\gamma_{n,1}(z'))<1/q\Bigr\}.
\end{aligned}
\]
Then $E_m(y,w)\in\mathcal{G}_w$, and $\xi\in E_m(y,w)$ if and only if the
restrictions of $\{\gamma_{n,1}\}_{n\in\mathbb{N}}$ to
$\overline{B}_Y(y,1/m)$ are uniformly equicontinuous.
By Definition~\ref{def:Julia-path}, if $y\notin J_\xi$, then
$\{\gamma_{n,1}\}_{n\in\mathbb N}$ is equicontinuous on a neighborhood
of $y$. Choosing $m$ so that $\overline{B}_Y(y,1/m)$ is contained in
this neighborhood, compactness of this closed ball gives uniform
equicontinuity on it, and hence $\xi\in E_m(y,w)$.
Conversely, membership in $E_m(y,w)$ implies equicontinuity on
$B_Y(y,1/m)$, so $y\notin J_\xi$.
Therefore,
\[
\bigl\{\xi=(\gamma_n,x_n)_{n\in\mathbb{N}}\in\Xi_w(S_\tau):y\notin J_\xi\bigr\}
=\bigcup_{m\in\mathbb{N}}E_m(y,w).
\]
Hence
\(
\bigl\{\xi=(\gamma_n,x_n)_{n\in\mathbb{N}}\in\Xi_w(S_\tau):y\in J_\xi\bigr\}
\in\mathcal{G}_w.
\)

Suppose now that
\[
\tilde{\tau}_w\Bigl(
\bigl\{\xi=(\gamma_n,x_n)_{n\in\mathbb{N}}\in\Xi_w(S_\tau):y\in J_\xi\bigr\}
\Bigr)=0.
\]
By Lemma~\ref{lem:pt-equicontinuity-M}, it is sufficient to prove that, for every
$\phi\in C(\mathbb{Y})$, the family
$\{M_\tau^n\phi\}_{n\in\mathbb{N}_0}$ is equicontinuous at $(y,w)$.
Fix $\phi\in C(\mathbb{Y})$ and $\varepsilon>0$. If $\phi\equiv0$, there is
nothing to prove. We therefore suppose that $\|\phi\|_\infty>0$ and set
\[
M:=\|\phi\|_\infty,
\qquad
\eta:=\min\left\{1,\frac{\varepsilon}{6M}\right\}.
\]
The uniform continuity of $\phi$ on $\mathbb{Y}$ gives $\delta_0>0$ such that
\[
d_{\mathbb{Y}}\bigl((z,v),(z',v')\bigr)<\delta_0
\quad\Longrightarrow\quad
|\phi(z,v)-\phi(z',v')|<\frac{\varepsilon}{3}
\]
for every $(z,v),(z',v')\in\mathbb{Y}$.

The sets $E_m(y,w)$ are increasing in $m$, and
\[
\tilde{\tau}_w\left(\bigcup_{m\in\mathbb{N}}E_m(y,w)\right)=1.
\]
We may therefore choose $m\in\mathbb{N}$ such that
$\tilde{\tau}_w(E_m(y,w))\ge1-\eta/2$. For
$\ell\in\mathbb{N}$ with $\ell\ge m$, we set
\[
\begin{aligned}
E_{m,\ell}(y,w)
&:=E_m(y,w)\cap\bigcap_{n\in\mathbb{N}}
\Bigl\{\xi=(\gamma_j,x_j)_{j\in\mathbb{N}}\in\Xi_w(S_\tau):{}\\
&\hspace{34mm}
\max_{z\in\overline{B}_Y(y,1/\ell)}
d_Y(\gamma_{n,1}(z),\gamma_{n,1}(y))<\delta_0\Bigr\}.
\end{aligned}
\]
The sets $E_{m,\ell}(y,w)$ belong to $\mathcal{G}_w$ and are increasing in
$\ell$. The uniform equicontinuity on $\overline{B}_Y(y,1/m)$ gives
\[
E_m(y,w)
=
\bigcup_{\substack{\ell\in\mathbb{N}\\ \ell\ge m}}
E_{m,\ell}(y,w).
\]
Hence we may choose $\ell\in\mathbb{N}$ with $\ell\ge m$ such that
$\tilde{\tau}_w(E_{m,\ell}(y,w))\ge1-\eta$. We set
$V_Y:=B_Y(y,1/\ell)$.

By Assumption~\ref{ass:standing-cooperation}\ref{ass:update-equicontinuity},
there exists an open neighborhood $V_W^{(1)}$ of $w$ such that
\[
d_W(va^{(n)},wa^{(n)})<\delta_0
\]
for every $v\in V_W^{(1)}$, every $n\in\mathbb{N}$, and every
$a^{(n)}\in X^n$.
By Assumption~\ref{ass:standing-cooperation}\ref{ass:prefix-continuity},
there exists an open neighborhood $V_W^{(2)}$ of $w$ such that
\[
\sup_{n\in\mathbb{N}}
\sum_{a^{(n)}\in X^n}
\left|
\mathbf{P}_v([a^{(n)}])-\mathbf{P}_w([a^{(n)}])
\right|
<\frac{\varepsilon}{3M}
\]
for every $v\in V_W^{(2)}$. We set
$V_W:=V_W^{(1)}\cap V_W^{(2)}$.

Let $(y',v)\in V_Y\times V_W$ and $n\in\mathbb{N}$.
By Lemmas~\ref{lem:Mn-formula}
and~\ref{lem:Mn-integral-representation},
extending the finite sums to $X^n$ by assigning probability zero to
inadmissible words gives
\[
\begin{aligned}
&\bigl|(M_\tau^n\phi)(y',v)-(M_\tau^n\phi)(y,w)\bigr|\\
&\quad\le
M\sum_{a^{(n)}\in X^n}
\left|
\mathbf{P}_v([a^{(n)}])-\mathbf{P}_w([a^{(n)}])
\right|\\
&\qquad+
\int_{\Xi_w(S_\tau)}
\left|
\phi\bigl(\gamma_{n,1}(y'),vx^{(n)}\bigr)
-\phi\bigl(\gamma_{n,1}(y),wx^{(n)}\bigr)
\right|\,d\tilde{\tau}_w(\xi),
\end{aligned}
\]
where $x^{(n)}=(x_1,\ldots,x_n)$.
Assumption~\ref{ass:standing-cooperation}\ref{ass:prefix-continuity} gives
\[
M\sum_{a^{(n)}\in X^n}
\left|\mathbf{P}_v([a^{(n)}])-\mathbf{P}_w([a^{(n)}])\right|
<\frac{\varepsilon}{3}.
\]
For every $\xi\in E_{m,\ell}(y,w)$, the inclusions
$y'\in V_Y\subset\overline{B}_Y(y,1/\ell)$ and $v\in V_W^{(1)}$ give
\[
d_Y\bigl(\gamma_{n,1}(y'),\gamma_{n,1}(y)\bigr)<\delta_0,
\qquad
d_W(vx^{(n)},wx^{(n)})<\delta_0.
\]
The integrand is therefore less than $\varepsilon/3$ on
$E_{m,\ell}(y,w)$ and is at most $2M$ on its complement. It follows that
the integral is at most
\[
\frac{\varepsilon}{3}
+2M\tilde{\tau}_w\bigl(\Xi_w(S_\tau)\setminus E_{m,\ell}(y,w)\bigr)
\le
\frac{\varepsilon}{3}+2M\eta
\le
\frac{2\varepsilon}{3}.
\]
Consequently,
\[
\bigl|(M_\tau^n\phi)(y',v)-(M_\tau^n\phi)(y,w)\bigr|<\varepsilon
\]
for every $(y',v)\in V_Y\times V_W$ and every $n\in\mathbb{N}$. Since
$M_\tau^0\phi=\phi$ is continuous, the family
$\{M_\tau^n\phi\}_{n\in\mathbb{N}_0}$ is equicontinuous at $(y,w)$.
Lemma~\ref{lem:pt-equicontinuity-M} now gives
$(y,w)\in F_{\mathrm{pt}}^{0}(M_\tau^*)$.
\end{proof}

Using the lemmas established above, we now prove the following theorem.
This result generalizes \cite[Theorem~3.14]{MR2747724} and
\cite[Proposition~3.11]{MR4002398}.

\begin{theorem}[Cooperation Principle~I]
\label{thm:CP-Fatou}
Let $\lambda$ be a finite Borel measure on $Y$.
Suppose that $S_\tau$ satisfies
Assumption~\ref{ass:standing-cooperation} and that
$J_{\ker,w}(S_\tau)=\emptyset$ for every $w\in W$.
Then the following assertions hold.
\begin{enumerate}[label=\textup{(\arabic*)}]
\item
$F_{\mathrm{meas}}(M_\tau^{*})=\mathfrak{M}_1(\mathbb{Y})$.

\item
For every $w\in W$ and for $\tilde{\tau}_w$-almost every
$\xi\in\Xi_w(S_\tau)$,
we have $\lambda(J_\xi)=0$.
\end{enumerate}
\end{theorem}

\begin{proof}
We set $\Omega:=(\mathrm{CM}(Y)\times X)^{\mathbb N}$ and equip it
with its product Borel $\sigma$-algebra.
For each $w\in W$, we extend $\tilde\tau_w$ to $\Omega$ by setting
\[
\tilde\tau_w(B):=\tilde\tau_w(B\cap\Xi_w(S_\tau)),
\qquad B\in\mathcal B(\Omega).
\]

We first establish the measurability required below.
We set $\mathcal A:=\{(y,\xi)\in Y\times\Omega:y\notin J_\xi\}$.
Choose a countable basis $(V_j)_{j\in\mathbb N}$ of nonempty open balls
in $Y$, and for each $j\in\mathbb N$, we set
\[
E_j:=
\left\{
\xi=(\gamma_k,x_k)_{k\in\mathbb N}\in\Omega:
\{\gamma_{n,1}|_{\overline V_j}\}_{n\in\mathbb N}
\text{ is uniformly equicontinuous}
\right\}.
\]
Then
\[
E_j=
\bigcap_{k\in\mathbb N}
\bigcup_{r\in\mathbb N}
\bigcap_{n\in\mathbb N}
\left\{
\xi\in\Omega:
\max_{\substack{z,z'\in\overline V_j\\ d_Y(z,z')\le1/r}}
d_Y(\gamma_{n,1}(z),\gamma_{n,1}(z'))<1/k
\right\}.
\]
For each $j,n,r\in\mathbb N$, the map
\[
\xi\longmapsto
\max_{\substack{z,z'\in\overline V_j\\ d_Y(z,z')\le1/r}}
d_Y(\gamma_{n,1}(z),\gamma_{n,1}(z'))
\]
is Borel measurable by the same argument as in
Lemma~\ref{lem:zero-probability-implies-Fpt}.
Hence $E_j$ is Borel measurable.
By Definition~\ref{def:Julia-path},
\[
\mathcal A=\bigcup_{j\in\mathbb N}(V_j\times E_j).
\]
Indeed, if $y\notin J_\xi$, we can choose $V_j$ containing $y$
whose closure lies in a neighborhood on which
$\{\gamma_{n,1}\}_{n\in\mathbb N}$ is equicontinuous.
Compactness of $\overline V_j$ then gives $\xi\in E_j$.
Conversely, if $y\in V_j$ and $\xi\in E_j$, then
$\{\gamma_{n,1}\}_{n\in\mathbb N}$ is equicontinuous on the
neighborhood $V_j$ of $y$, and hence $y\notin J_\xi$.
Thus $\mathcal A$ is Borel measurable.

For $(y,w)\in\mathbb Y$, we define
\[
q(y,w)
:=
\tilde\tau_w
\{\xi\in\Omega:y\notin J_\xi\}.
\]
We first show that $q(y,w)>0$ for every $(y,w)\in\mathbb Y$.
Fix $(y,w)\in\mathbb Y$.
By the kernel-emptiness assumption and
Definition~\ref{def:kernel-julia-set}, there exist
$m\in\mathbb N$, an admissible word
$a^{(m)}=(a_1,\ldots,a_m)\in X_{w,m}$, and maps
$g_j\in\Gamma_{a_j}$, $1\le j\le m$, such that
\[
g_m\circ\cdots\circ g_1(y)\in F_{wa^{(m)}}(S_\tau).
\]
Since this Fatou set is open, continuity of composition and evaluation
gives open neighborhoods $O_j$ of $g_j$ in $\mathrm{CM}(Y)$ such that
\[
h_m\circ\cdots\circ h_1(y)\in F_{wa^{(m)}}(S_\tau)
\qquad\text{whenever }h_j\in O_j\ (1\le j\le m).
\]
We set
\[
C:=
\left\{
\xi=(\gamma_k,x_k)_{k\in\mathbb N}\in\Xi_w(S_\tau):
x_j=a_j,\ \gamma_j\in O_j\ (1\le j\le m)
\right\}.
\]
By Definition~\ref{def:Pwtilde},
\[
\tilde\tau_w(C)
=
\mathbf P_w([a^{(m)}])
\prod_{j=1}^m\tau_{a_j}(O_j)>0,
\]
because $a^{(m)}$ is admissible and
$g_j\in\supp(\tau_{a_j})$ for each $j$.
For every $\xi\in C$, we have
$\gamma_{m,1}(y)\in F_{wa^{(m)}}(S_\tau)$ and
$\gamma_{m+k,m+1}\in H_{wa^{(m)}}(S_\tau)$ for every $k\in\mathbb N$.
Hence $\{\gamma_{m+k,m+1}\}_{k\in\mathbb N}$ is equicontinuous on a
neighborhood of $\gamma_{m,1}(y)$.
Since
$\gamma_{m+k,1}=\gamma_{m+k,m+1}\circ\gamma_{m,1}$,
continuity of $\gamma_{m,1}$ implies that
$\{\gamma_{m+k,1}\}_{k\in\mathbb N}$ is equicontinuous on a
neighborhood of $y$.
Adding the finitely many continuous maps
$\gamma_{1,1},\ldots,\gamma_{m,1}$ preserves equicontinuity.
Thus $y\notin J_\xi$ for every $\xi\in C$, and therefore
$q(y,w)\ge\tilde\tau_w(C)>0$.

We next show that $q$ is lower semicontinuous on $\mathbb Y$.
Fix $w\in W$ and suppose that $y_k\to y$.
Since $Y\setminus J_\xi$ is open for every $\xi\in\Omega$,
\[
\mathbf1_{\{y\notin J_\xi\}}
\le
\liminf_{k\to\infty}\mathbf1_{\{y_k\notin J_\xi\}},
\]
and Fatou's lemma yields
\[
q(y,w)\le\liminf_{k\to\infty}q(y_k,w).
\]

To control the dependence on the state variable, for $v,w\in W$ we set
\[
D(v,w):=
\sup_{n\in\mathbb N}
\sum_{a^{(n)}\in X^n}
\left|
\mathbf P_v([a^{(n)}])-\mathbf P_w([a^{(n)}])
\right|.
\]
For probability measures $\nu_1,\nu_2$ on $\Omega$, we write
\[
d_{\mathrm{TV}}(\nu_1,\nu_2)
:=
\sup_{B\in\mathcal B(\Omega)}
|\nu_1(B)-\nu_2(B)|.
\]
We claim that
\[
d_{\mathrm{TV}}(\tilde\tau_v,\tilde\tau_w)
\le\frac12D(v,w).
\]
For $n\in\mathbb N$, let
\[
\Pi_n:\Omega\to(\mathrm{CM}(Y)\times X)^n
\]
be the projection onto the first $n$ coordinates.
For each $a^{(n)}\in X^n$, let $\rho_{a^{(n)}}$ be the distribution of
$((g_1,a_1),\ldots,(g_n,a_n))$ when
$(g_1,\ldots,g_n)$ has distribution
$\tau_{a_1}\otimes\cdots\otimes\tau_{a_n}$.
For every Borel set
$B\subset(\mathrm{CM}(Y)\times X)^n$,
Definition~\ref{def:Pwtilde} gives
\[
\tilde\tau_v(\Pi_n^{-1}(B))
=
\sum_{a^{(n)}\in X^n}
\mathbf P_v([a^{(n)}])\,\rho_{a^{(n)}}(B).
\]
Since $\rho_{a^{(n)}}$ does not depend on $v$, the differences of the
coefficients sum to zero, and $0\le\rho_{a^{(n)}}(B)\le1$, we obtain
\[
\left|
\tilde\tau_v(\Pi_n^{-1}(B))
-\tilde\tau_w(\Pi_n^{-1}(B))
\right|
\le
\frac12
\sum_{a^{(n)}\in X^n}
\left|
\mathbf P_v([a^{(n)}])-\mathbf P_w([a^{(n)}])
\right|
\le\frac12D(v,w).
\]
The sets $\Pi_n^{-1}(B)$, with $n\in\mathbb N$ and $B$ Borel,
form an algebra generating $\mathcal B(\Omega)$.
A standard approximation argument with respect to the finite measure
$\tilde\tau_v+\tilde\tau_w$ therefore extends this estimate to every
Borel subset of $\Omega$, proving the claim.
Consequently,
\[
|q(y,v)-q(y,w)|\le\frac12D(v,w)
\qquad(y\in Y).
\]
By Assumption~\ref{ass:standing-cooperation}\ref{ass:prefix-continuity},
$D(v,w)\to0$ as $v\to w$.
Hence, if $(y_k,w_k)\to(y,w)$, then
\[
\liminf_{k\to\infty}q(y_k,w_k)
\ge
\liminf_{k\to\infty}q(y_k,w)
\ge q(y,w).
\]
Thus $q$ is lower semicontinuous on $\mathbb Y$.
Since $\mathbb Y$ is compact and $q$ is strictly positive, there exists
$\delta>0$ such that
\[
q(y,w)\ge\delta
\qquad\text{for every }(y,w)\in\mathbb Y.
\]

We now use this uniform lower bound to show that the pathwise Julia
probability vanishes.
Fix $(y,w)\in\mathbb Y$ and set
\[
A_y:=\{\xi\in\Omega:y\in J_\xi\}.
\]
By the joint measurability established above, $A_y\in\mathcal B(\Omega)$.

For each $j\in\mathbb N$, let
\[
\pi_j:\Omega\to\mathrm{CM}(Y)\times X
\]
denote the $j$th coordinate map.
For $\xi\in\Omega$, we write
\[
\pi_j(\xi)=\bigl(\gamma_j(\xi),x_j(\xi)\bigr).
\]
For $n\in\mathbb N$, we set
\[
\mathcal H_n
:=
\sigma\!\left(
\pi_j^{-1}(B):
1\le j\le n,\ 
B\in\mathcal B(\mathrm{CM}(Y)\times X)
\right)
\subset\mathcal B(\Omega).
\]
We also set
\[
x^{(n)}(\xi)
:=
\bigl(x_1(\xi),\ldots,x_n(\xi)\bigr),
\qquad
\gamma_{n,1}(\xi)
:=
\gamma_n(\xi)\circ\cdots\circ\gamma_1(\xi),
\]
and define the shift $\theta^n:\Omega\to\Omega$ by
\[
\theta^n\xi
:=
\bigl(\pi_{n+k}(\xi)\bigr)_{k\in\mathbb N}.
\]

By Definition~\ref{def:Julia-path}, for every $\xi\in\Omega$ and
$n\in\mathbb N$,
\[
y\in J_\xi
\quad\Longrightarrow\quad
\gamma_{n,1}(\xi)(y)\in J_{\theta^n\xi}.
\]
Indeed, if
\[
\gamma_{n,1}(\xi)(y)\notin J_{\theta^n\xi},
\]
then the equicontinuity of the compositions associated with
$\theta^n\xi$, together with the continuity of the finite prefix
$\gamma_{n,1}(\xi)$, would imply equicontinuity of
$\{\gamma_{k,1}(\xi)\}_{k\in\mathbb N}$ on a neighborhood of $y$.

By the probability kernels in Definition~\ref{def:Pwtilde}, conditionally
on $\mathcal H_n$, the shifted sequence $\theta^n\xi$ has distribution
$\tilde\tau_{w x^{(n)}(\xi)}$.
Hence, for every bounded Borel measurable function
$\Psi:Y\times\Omega\to\mathbb R$,
\[
\mathbf E_{\tilde\tau_w}
\left[
\Psi\bigl(\gamma_{n,1}(\xi)(y),\theta^n\xi\bigr)
\mid\mathcal H_n
\right]
=
\int_\Omega
\Psi\bigl(\gamma_{n,1}(\xi)(y),\eta\bigr)
\,d\tilde\tau_{w x^{(n)}(\xi)}(\eta)
\]
for $\tilde\tau_w$-almost every $\xi\in\Omega$.
The right-hand side is understood as an $\mathcal H_n$-measurable
function of $\xi$.

By the joint measurability established above, the function
\[
\Psi(z,\eta):=\mathbf1_{\{z\in J_\eta\}}
\]
is Borel measurable on $Y\times\Omega$.
The preceding implication gives
\[
\mathbf1_{A_y}(\xi)
\le
\mathbf1_{\{\gamma_{n,1}(\xi)(y)\in J_{\theta^n\xi}\}}.
\]
Taking conditional expectations with respect to $\mathcal H_n$, we obtain
\[
\begin{aligned}
\mathbf E_{\tilde\tau_w}
[\mathbf1_{A_y}\mid\mathcal H_n](\xi)
&\le
\int_\Omega
\mathbf1_{\{\gamma_{n,1}(\xi)(y)\in J_\eta\}}
\,d\tilde\tau_{w x^{(n)}(\xi)}(\eta)\\
&=
1-q\bigl(\gamma_{n,1}(\xi)(y),w x^{(n)}(\xi)\bigr)
\le1-\delta
\end{aligned}
\]
for $\tilde\tau_w$-almost every $\xi\in\Omega$.

The sequence $(\mathcal H_n)_{n\in\mathbb N}$ is increasing, and the
coordinate maps generate the product Borel $\sigma$-algebra on $\Omega$.
Hence
\[
\sigma\left(\bigcup_{n\in\mathbb N}\mathcal H_n\right)
=
\mathcal B(\Omega).
\]
Since $A_y\in\mathcal B(\Omega)$, L\'evy's upward theorem yields
\[
\mathbf E_{\tilde\tau_w}
[\mathbf1_{A_y}\mid\mathcal H_n]
\longrightarrow
\mathbf1_{A_y}
\qquad
\tilde\tau_w\text{-almost surely}.
\]
It follows that
\[
\mathbf1_{A_y}\le1-\delta
\qquad
\tilde\tau_w\text{-almost surely}.
\]
Since $\mathbf1_{A_y}$ takes only the values $0$ and $1$ and
$\delta>0$, we conclude that
\[
\tilde\tau_w(A_y)=0.
\]
As $(y,w)\in\mathbb Y$ was arbitrary,
\[
\tilde\tau_w
\{\xi\in\Xi_w(S_\tau):y\in J_\xi\}
=0
\qquad
\text{for every }(y,w)\in\mathbb Y.
\]

Lemma~\ref{lem:zero-probability-implies-Fpt} now gives
\[
F_{\mathrm{pt}}^0(M_\tau^*)=\mathbb Y.
\]
Applying Lemma~\ref{lem:Fmeas-Fpt-equivalence}, we obtain
\[
F_{\mathrm{meas}}(M_\tau^*)
=
\mathfrak M_1(\mathbb Y),
\]
which proves assertion~\textup{(1)}.

Finally, fix $w\in W$.
By the joint measurability of
$\{(y,\xi)\in Y\times\Omega:y\in J_\xi\}$ and Tonelli's theorem,
\[
\begin{aligned}
\int_{\Xi_w(S_\tau)}
\lambda(J_\xi)\,d\tilde\tau_w(\xi)
&=
\int_{\Xi_w(S_\tau)}
\int_Y
\mathbf1_{\{y\in J_\xi\}}
\,d\lambda(y)\,d\tilde\tau_w(\xi)\\
&=
\int_Y
\tilde\tau_w
\{\xi\in\Xi_w(S_\tau):y\in J_\xi\}
\,d\lambda(y)
=0.
\end{aligned}
\]
Since $\lambda(J_\xi)\ge0$, it follows that
$\lambda(J_\xi)=0$ for $\tilde\tau_w$-almost every
$\xi\in\Xi_w(S_\tau)$.
This proves assertion~\textup{(2)}.
\end{proof}

We conclude this section by combining the absence of the jump phenomenon under
$\varphi$-irreducibility with the Cooperation Principle~I.
In a $\varphi$-irreducible system, the emptiness of kernel Julia sets on a set
of positive $\varphi$-measure propagates to all states, thereby ruling out
jumps and restoring cooperation.
This yields the following theorem.

\begin{theorem}[Cooperation under $\varphi$-irreducibility]
\label{thm:CP-Fatou-phi-irreducible}
Let $\lambda$ be a finite Borel measure on $Y$.
Suppose that $S_\tau$ satisfies Assumption~\ref{ass:standing-cooperation}.
Let $\varphi$ be a $\sigma$-finite measure on $(W,\mathcal W)$ and assume that
$S_\tau$ is $\varphi$-irreducible.

Set
\[
A:=\{\,w\in W:J_{\ker,w}(S_\tau)=\emptyset\,\}.
\]
Assume that $A\in\mathcal W$ and that $\varphi(A)>0$.
Then the conclusions of Theorem~\ref{thm:CP-Fatou} hold; namely:
\begin{enumerate}[label=\textup{(\arabic*)}]
\item
$F_{\mathrm{meas}}(M_\tau^{*})=\mathfrak{M}_1(\mathbb{Y})$.

\item
For every $w\in W$ and for $\tilde{\tau}_w$-almost every
$\xi\in\Xi_w(S_\tau)$,
we have $\lambda(J_\xi)=0$.
\end{enumerate}
\end{theorem}

\begin{proof}
By Lemma~\ref{lemma:phi-irreducible-propagation-kernel-empty}, the
$\varphi$-irreducibility of $S_\tau$, together with $A\in\mathcal W$
and $\varphi(A)>0$, implies that
\[
J_{\ker,w}(S_\tau)=\emptyset\qquad\text{for all }w\in W.
\]
Applying Theorem~\ref{thm:CP-Fatou} proves assertions~\textup{(1)}
and~\textup{(2)}.
\end{proof}

\section{Examples}
\label{sec:examples}
In this section, we present examples of RSCCs motivated by reinforcement
and feedback mechanisms. Example~\ref{ex:linear-reinforcement-rscc}
satisfies Assumption~\ref{ass:standing-cooperation} and exhibits an
emptiness jump. Example~\ref{ex:clipped-reinforcement-rscc} satisfies
all the hypotheses of Theorem~\ref{thm:CP-Fatou}.
Example~\ref{ex:feedback} illustrates a general feedback construction.

As shown in Example~\ref{ex:rscc-gdms-mrds}, finite GDMSs can be
realized as RSCCs. By contrast, the greater generality of RSCCs permits
phenomena that cannot occur in finite GDMSs.
See Examples~\ref{ex:kernel-julia-jump} and~\ref{ex:linear-reinforcement-rscc} for such instances.

In what follows, we present examples motivated by reinforcement learning
and feedback mechanisms.
These illustrate that the RSCC framework provides a natural and flexible setting
for dynamical systems in which the selection probabilities depend on the evolving state.

\begin{example}
\label{ex:linear-reinforcement-rscc}
Let $W:=[0,1]\subset \mathbb{R}$ be endowed with the Euclidean metric and
$\mathcal W:=\mathcal B([0,1])$.
Let $X=\{0,1\}\subset \mathbb{R}$ be endowed with the discrete $\sigma$-algebra
$\mathcal X:=\mathcal{P}(X)$.
For $p\in W$, define a probability measure $P(p,\cdot)$ on $X$ by
\[
P(p,\{1\}):=p,
\qquad
P(p,\{0\}):=1-p.
\]
Fix $\alpha\in(0,1)$ and define the update map
$u:W\times X\to W$ by
\[
u(p,x):=(1-\alpha)p+\alpha x.
\]
Then $((W,\mathcal{W}),(X,\mathcal{X}),u,P)$ is an RSCC satisfying Assumption~\ref{ass:standing-cooperation}.

Indeed, $W$ is compact and $X$ is finite. For every $p,q\in W$ and
$a^{(n)}\in X^n$,
\[
|pa^{(n)}-qa^{(n)}|=(1-\alpha)^n|p-q|,
\]
which proves Assumption~\ref{ass:standing-cooperation}\ref{ass:update-equicontinuity}.
To verify Assumption~\ref{ass:standing-cooperation}\ref{ass:prefix-continuity}, we set
\[
D_n(p,q):=\sum_{a^{(n)}\in X^n}
\left|\mathbf P_p([a^{(n)}])-\mathbf P_q([a^{(n)}])\right|
\quad(n\ge1),\qquad D_0(p,q):=0.
\]
The identity
\[
\mathbf P_p([a^{(n)},x])
=\mathbf P_p([a^{(n)}])P(pa^{(n)},\{x\})
\]
and the triangle inequality give
\[
\begin{aligned}
D_{n+1}(p,q)
&\le D_n(p,q)
 +\sum_{a^{(n)}\in X^n}\mathbf P_q([a^{(n)}])
 \sum_{x\in X}|P(pa^{(n)},\{x\})-P(qa^{(n)},\{x\})|\\
&=D_n(p,q)+2(1-\alpha)^n|p-q|.
\end{aligned}
\]
For $n=0$, the same estimate is $D_1(p,q)=2|p-q|$.
Consequently,
\[
\sup_{n\ge1}D_n(p,q)
\le 2|p-q|\sum_{j=0}^{\infty}(1-\alpha)^j
=\frac{2}{\alpha}|p-q|.
\]
This proves Assumption~\ref{ass:standing-cooperation}\ref{ass:prefix-continuity}.

Let $(p_n)_{n\in\mathbb N_0}$ and $(\xi_n)_{n\in\mathbb N}$ denote the
associated state and index processes, respectively.
By definition of the update map,
\[
p_{n+1}=(1-\alpha)p_n+\alpha\xi_{n+1}
\qquad(n\in\mathbb N_0).
\]
Thus the state is updated by moving $p_n$ toward the chosen index
$\xi_{n+1}$.
In particular, if the index $1$ (resp.\ $0$) is chosen,
then the probability of choosing $1$ (resp.\ $0$) at the next step increases.
The parameter $\alpha$ determines how far the state moves toward the chosen index.
Hence the system may be viewed as a simple reinforcement-type mechanism.

For $p_0\in(0,1)$, the index process in this example is not a Markov
chain of any finite order. Indeed, we have
\[
p_n=(1-\alpha)^n p_0+
\alpha\sum_{j=1}^n(1-\alpha)^{n-j}\xi_j,
\qquad
\mathbf P_{p_0}(\xi_{n+1}=1\mid\xi_1,\ldots,\xi_n)=p_n.
\]
Every finite index sequence has positive probability, since all finite
updates from $p_0\in(0,1)$ remain in $(0,1)$.
For any $k\ge1$, choose two sequences of length $k+1$ that differ only
in their first index. Their last $k$ indices coincide, but their
conditional probabilities of selecting $1$ at the next step differ by
$\alpha(1-\alpha)^k>0$. Hence the index process is not a Markov chain
of order $k$. Nevertheless, the state process $(p_n)_{n\in \mathbb{N}_0}$ is Markov.

\medskip
For each index $x\in X$, we define $\tau_x := \delta_{f_x}$ and $\Gamma_x := \supp (\tau_x) = \{f_x\}$, where
\[
f_0(z):=z^2,
\qquad
f_1(z):=\frac{z^2}{2}.
\]
Then $S_\tau := ((W,\mathcal{W}),(X,\mathcal{X}),u,P,\{\Gamma_x\}_{x\in X})$ is an RSCC on $\widehat{\mathbb{C}}$.

For $0<p<1$, both indices are admissible and every finite sequence of
updates remains in $(0,1)$. Thus
\[
H_p(S_\tau)
=
G
:=
\langle f_0,f_1\rangle
:=
\bigl\{
h_n\circ\cdots\circ h_1:
n\in\mathbb N,\ h_j\in\{f_0,f_1\}\ (1\le j\le n)
\bigr\}.
\]
We denote by $F(G)$ the set of points on whose neighborhood $G$ is
equicontinuous and set $J(G):=\widehat{\mathbb C}\setminus F(G)$.
The estimates $|f_x(z)|\le |z|^2$ for $|z|<1$ and
$|f_x(z)|\ge |z|^2/2$ for $|z|>2$ imply uniform convergence to $0$
and $\infty$, respectively, on compact subsets of these regions as
the composition length tends to infinity. Since there are only finitely
many compositions of bounded length,
\[
J(G)\subset\{z\in\mathbb C:1\le |z|\le2\}.
\]
If $|z|\le1$, then $|f_1(z)|<1$; if $|z|>1$, some iterate of $f_0$
sends $z$ outside the closed disk of radius $2$. Also $\infty\in F(G)$
and $f_0(\infty)=\infty$. Therefore every point can be mapped into
$F(G)$ by an admissible composition, so $J_{\ker,p}(S_\tau)=\emptyset$
for $0<p<1$. At $p=0$ only $f_0$ is admissible, whereas at $p=1$
only $f_1$ is admissible. Both states are fixed under their admissible
updates. Hence
\[
J_{\ker,p}(S_\tau)
=
\begin{cases}
J(f_0)\ne \emptyset  & \text{if } p=0,\\
\emptyset & \text{if } 0<p<1,\\
J(f_1)\ne \emptyset & \text{if } p=1.
\end{cases}
\]
Fix $p_0\in(0,1)$ and take $x_n=0$ for every $n\in\mathbb N$.
Every finite prefix of this sequence has positive probability, so the
sequence belongs to $\operatorname{supp}(\mathbf P_{p_0})$.
The corresponding states satisfy
\[
p_n=(1-\alpha)^n p_0\longrightarrow0,
\qquad
J_{\ker,p_n}(S_\tau)=\emptyset\quad(n\ge1),
\qquad
J_{\ker,0}(S_\tau)=J(f_0)\ne\emptyset.
\]
Thus the system exhibits an emptiness jump in the sense of
Definition~\ref{def:kernel-julia-jump}.

We further note that assertion~\textup{(1)} of
Theorem~\ref{thm:CP-Fatou} fails in this example, and that
assertion~\textup{(2)} also fails for a suitable finite Borel measure.
At the state $p=0$, the dynamics is the deterministic iteration of
$f_0(z)=z^2$. Let $\chi\in C(\widehat{\mathbb C})$ satisfy
$\chi(1)\ne\chi(2)$, and define
\[
\phi(z,p):=\chi(z).
\]
For $z_n:=2^{1/2^n}$, we have $z_n\to1$ and
$f_0^n(z_n)=2$, whereas $f_0^n(1)=1$. Hence
\[
(M_\tau^n\phi)(z_n,0)=\chi(2),
\qquad
(M_\tau^n\phi)(1,0)=\chi(1).
\]
Therefore $\{M_\tau^n\phi\}_{n\in\mathbb N}$ is not equicontinuous at
$(1,0)$. By Lemmas~\ref{lem:pt-equicontinuity-M}
and~\ref{lem:Fmeas-Fpt-equivalence},
\[
F_{\mathrm{meas}}(M_\tau^*)
\ne
\mathfrak M_1(\mathbb Y).
\]

Moreover, let
\[
\xi_0:=((f_0,0),(f_0,0),\ldots).
\]
Then $\tilde\tau_0(\{\xi_0\})=1$ and
\(
J_{\xi_0}=J(f_0)=\{z\in\mathbb C:|z|=1\}.
\)
Taking $\lambda:=\delta_1$, the Dirac measure at $1$, we obtain
\[
\lambda(J_{\xi_0})=1\ne0.
\]
Hence assertion~\textup{(2)} of Theorem~\ref{thm:CP-Fatou} does not hold
for this choice of $\lambda$ and $p=0$.

Thus, although Assumption~\ref{ass:standing-cooperation} is satisfied,
both conclusions of the theorem can fail when the kernel-emptiness
assumption is omitted.
Consequently, the kernel-emptiness assumption cannot in general be omitted.
\end{example}
In Example~\ref{ex:linear-reinforcement-rscc}, the kernel Julia sets at
$0$ and $1$ are nonempty.
In the next example, we restrict the state space and modify the updates
so that all the hypotheses of Theorem~\ref{thm:CP-Fatou} hold.

\begin{example}
\label{ex:clipped-reinforcement-rscc}
Fix $\alpha\in(0,1)$ and $\varepsilon\in(0,1/2)$.
Define
\[
W_\varepsilon:=[\varepsilon,1-\varepsilon]\subset\mathbb R.
\]
Let $(X,\mathcal X)$ be the same as in
Example~\ref{ex:linear-reinforcement-rscc}.
For $p\in W_\varepsilon$, define the transition probabilities by
\[
P_\varepsilon(p,\{1\})
:=
p,
\qquad
P_\varepsilon(p,\{0\})
:=
1-p.
\]
Define the update map
\(
u_\varepsilon:W_\varepsilon\times X\to W_\varepsilon
\)
by
\[
u_\varepsilon(p,x)
=
\begin{cases}
\varepsilon,
& (1-\alpha)p+\alpha x\le \varepsilon,\\[6pt]
(1-\alpha)p+\alpha x,
& \varepsilon < (1-\alpha)p+\alpha x <1-\varepsilon,\\[6pt]
1-\varepsilon,
& (1-\alpha)p+\alpha x\ge 1-\varepsilon,
\end{cases}
\]
for $(p,x)\in W_\varepsilon\times X$.
Then $((W_\varepsilon,\mathcal B(W_\varepsilon)),(X,\mathcal X),u_\varepsilon,P_\varepsilon)$
is an RSCC.
It satisfies Assumption~\ref{ass:standing-cooperation}.
Indeed, $W_\varepsilon$ is compact and $X$ is finite. The metric projection onto
$W_\varepsilon$ is $1$-Lipschitz, so
\[
|u_\varepsilon(p,x)-u_\varepsilon(q,x)|
\le(1-\alpha)|p-q|
\qquad(p,q\in W_\varepsilon,\ x\in X).
\]
Consequently,
\[
|pa^{(n)}-qa^{(n)}|\le(1-\alpha)^n|p-q|
\qquad(a^{(n)}\in X^n),
\]
where the iterated updates are taken with respect to $u_\varepsilon$.
This proves Assumption~\ref{ass:standing-cooperation}\ref{ass:update-equicontinuity}.
Writing $\mathbf P_p^\varepsilon$ for the associated index-sequence
measure, the same argument as in Example~\ref{ex:linear-reinforcement-rscc}
gives
\[
\sup_{n\ge1}\sum_{a^{(n)}\in X^n}
\left|\mathbf P_p^\varepsilon([a^{(n)}])
-\mathbf P_q^\varepsilon([a^{(n)}])\right|
\le\frac{2}{\alpha}|p-q|.
\]
Thus Assumption~\ref{ass:standing-cooperation}\ref{ass:prefix-continuity}
holds as well.

Let $f_x$, $\tau_x=\delta_{f_x}$, and $\Gamma_x=\{f_x\}$, $x\in X$,
be as in Example~\ref{ex:linear-reinforcement-rscc}.
Denote by $S_{\tau}^\varepsilon$ the associated RSCC on $\widehat{\mathbb{C}}$.
Since $P_\varepsilon(p,\{1\})\ge \varepsilon>0$ and
$P_\varepsilon(p,\{0\})\ge \varepsilon>0$ for all $p\in W_\varepsilon$,
both indices are always admissible.
For $G=\langle f_0,f_1\rangle$ as in
Example~\ref{ex:linear-reinforcement-rscc}, we therefore have
\[
H_p(S_\tau^\varepsilon)=G,\qquad
F_p(S_\tau^\varepsilon)=F(G),\qquad
J_p(S_\tau^\varepsilon)=J(G)
\qquad(p\in W_\varepsilon).
\]
The same argument as in Example~\ref{ex:linear-reinforcement-rscc}
shows that every point can be mapped into $F(G)$ by an admissible
composition. Hence
\[
J_{\ker,p}(S_{\tau}^\varepsilon)=\emptyset
\qquad \text{for all } p\in W_\varepsilon.
\]
Thus all the hypotheses of Theorem~\ref{thm:CP-Fatou} are satisfied.
In particular, with $\mathbb Y=\widehat{\mathbb C}\times W_\varepsilon$,
\[
F_{\mathrm{meas}}(M_\tau^*)=\mathfrak M_1(\mathbb Y).
\]
For every finite Borel measure $\lambda$ on $\widehat{\mathbb C}$
and every $p\in W_\varepsilon$, we also have $\lambda(J_\xi)=0$
for $\tilde\tau_p$-almost every $\xi\in\Xi_p(S_\tau^\varepsilon)$.
Here $M_\tau$ and $\tilde\tau_p$ are the operator and path measure
associated with $S_\tau^\varepsilon$.
\end{example}
In the preceding examples, the selection of indices was determined solely by the RSCC structure and was defined independently of the action of continuous maps on a compact metric space.
We now construct a model in which the choice of indices is influenced by the action of continuous maps on a compact metric space.
In other words, we introduce a system with feedback from the dynamics on the compact metric space.

\begin{example}
\label{ex:feedback}
Let $Y$ be a nonempty compact metric space and let $f,g \in \mathrm{CM}(Y)$ be distinct.
Define the state space by $W:=Y$ and $\mathcal W:=\mathcal B(W)$.
Define the index space by $X:=\{f,g\}$ and $\mathcal X:=\mathcal{P}(X)$.
Define
\[
\Gamma_{f}:=\{f\},
\qquad
\Gamma_{g}:=\{g\}.
\]
Fix a continuous function $\theta :Y\to[0,1]$.
Define a transition probability function $P:W\times\mathcal X\to[0,1]$ by
\[
P(y,\{f\})=\theta(y),
\qquad
P(y,\{g\})=1-\theta(y).
\]
Define
\[
u(y,f):=f(y),
\qquad
u(y,g):=g(y).
\]
Then $S=((W,\mathcal W),(X,\mathcal X),u,P,\{\Gamma_x\}_{x\in X})$
is an RSCC on $Y$.
If the initial state and the initial point in $Y$ coincide, then applying
the chosen map to both preserves this equality at every step.
The probability of choosing $f$ is then $\theta$ evaluated at the current
point of the orbit, which gives the stated feedback mechanism.
\end{example}

\begin{remark}
In the preceding example, the feedback is reflected immediately in the selection mechanism.
By a slight modification of the RSCC, more complex forms of feedback can be implemented.
For example, fix $k\ge2$ and a continuous function
$\Theta:Y^k\to[0,1]$. On $W=Y^k$, with the product topology and
its Borel $\sigma$-algebra, set
\[
\begin{aligned}
P((w_1,\ldots,w_k),\{f\})&=\Theta(w_1,\ldots,w_k),\\
P((w_1,\ldots,w_k),\{g\})&=1-\Theta(w_1,\ldots,w_k),
\end{aligned}
\]
and define
\[
u((w_1,\ldots,w_k),h)=(w_2,\ldots,w_k,h(w_k))
\qquad(h\in\{f,g\}).
\]
Keeping $X$, $\Gamma_f$, and $\Gamma_g$ as in Example~\ref{ex:feedback}
gives an RSCC whose selection probabilities can depend on the $k$
components stored in the state.
\end{remark}

\bigskip
\noindent\textbf{Acknowledgments.}
The author would like to express his sincere gratitude to
Professor Johannes Jaerisch and Professor Hiroshi Sumi
for valuable discussions and helpful comments.
The author is also grateful to the anonymous referees for their careful reading
of the manuscript and for their valuable comments and suggestions.
\bibliographystyle{alpha}
\bibliography{references}
\end{document}